%% file: neurips_2022.tex
\pdfoutput=1

\documentclass{article}
\usepackage[export]{adjustbox}
\usepackage{tikz}
\usepackage{subcaption}
\usepackage{answers}
\usepackage{xcolor}
\usepackage{setspace}
\usepackage{graphicx}
\usepackage[shortlabels]{enumitem}
\usepackage{multicol}
\usepackage{mathrsfs}
\usepackage{amsmath,amsthm,amssymb}
\usepackage{mathtools}
\usepackage[utf8]{inputenc}
\usepackage[english]{babel}
\usepackage{bbm}
\usepackage{breqn}
\usepackage{subfiles}
\usepackage{amsmath,amssymb}
\usepackage[english]{babel}
\usepackage[nottoc]{tocbibind}
\usepackage{bm}
\usepackage{hyperref}
\usepackage{algorithm}
\usepackage{thm-restate}
\usepackage{todonotes}

\newtheorem{theorem}{Theorem}[section]

\newtheorem{lemma}[theorem]{Lemma}

\newtheorem{fact}[theorem]{Fact}

\newcommand{\I}{\mathcal I}

\newcommand{\Z}{\mathbb{Z}}

\newcommand{\R}{\mathbb{R}}

\DeclareMathOperator*{\E}{\mathbb{E}}

\newcommand{\sign}{\text{sign}}

\newcommand{\lam}{\lambda}
\newcommand{\eps}{\varepsilon}
\newcommand{\1}{\mathbbm{1}}

\newcommand{\abs}[1]{|#1| }

\newcommand{\wh}{\widehat}
\newcommand{\from}{\leftarrow}
\renewcommand{\d}{\mathrm{d}}

\DeclarePairedDelimiterX{\kl}[2]{D_{\mathrm{KL}}(}{)}{%
  #1\,\delimsize\|\,#2%
}

\hypersetup{
    colorlinks=true,
    linkcolor=blue,
    filecolor=magenta,      
    urlcolor=cyan,
}

\newcommand{\lambdahat}{\hat{\lambda}}
\newcommand{\Normal}{\mathcal{N}}
\newcommand{\Exp}{\mathbb{E}}
\newcommand{\IQR}{\mathrm{IQR}}
\newcommand{\DTV}{d_\mathrm{TV}}
\newcommand{\DKL}{D_\mathrm{KL}}
\renewcommand{\DH}{d_\mathrm{H}}
\renewcommand{\vec}{}

\usepackage[preprint,nonatbib]{neurips_2022}

\usepackage[utf8]{inputenc} 
\usepackage[T1]{fontenc}    
\usepackage{hyperref}       
\usepackage{url}            
\usepackage{booktabs}       
\usepackage{amsfonts}       
\usepackage{nicefrac}       
\usepackage{microtype}      
\usepackage{xcolor}         

\title{Finite-Sample Maximum Likelihood Estimation of Location}

\author{%
  Shivam Gupta\\
  The University of Texas at Austin\\
  \texttt{shivamgupta@utexas.edu} \\
   \And
   Jasper C.H.~Lee \\
   University of Wisconsin--Madison \\
   \texttt{jasper.lee@wisc.edu} \\
   \And
   Eric Price \\
   The University of Texas at Austin \\
   \texttt{ecprice@cs.utexas.edu} \\
   \And
   Paul Valiant \\
   Purdue University \\
   \texttt{pvaliant@gmail.com} \\
}

\begin{document}

\maketitle

\begin{abstract}
  We consider 1-dimensional location estimation, where we estimate a
  parameter $\lambda$ from $n$ samples $\lambda + \eta_i$, with each $\eta_i$
  drawn i.i.d.~from a known distribution $f$.  For fixed $f$ the maximum-likelihood estimate (MLE) is well-known
  to be optimal \emph{in the limit} as $n \to \infty$: it is asymptotically normal with variance matching
  the Cram\'er-Rao lower bound of $\frac{1}{n\I}$, where $\I$ is the
  Fisher information of $f$.  However, this bound does not hold
  for finite $n$, or when $f$ varies with $n$.  We show for arbitrary
  $f$ and $n$ that one can recover a similar theory based on the
  Fisher information of a \emph{smoothed} version of $f$, where the
  smoothing radius decays with $n$.
\end{abstract}

\input{intro}

\input{literature}

\input{score}

\input{localMLE}

\input{globalMLE}

\input{lowerbound}

\input{experiments}

\input{futureWork}

\section*{Acknowledgements}

Shivam Gupta and Eric Price are supported by NSF awards CCF-2008868, CCF-1751040 (CAREER), and the NSF AI Institute for Foundations of Machine Learning (IFML).
Jasper C.H.~Lee is supported in part by the generous funding of a Croucher Fellowship for Postdoctoral Research and by NSF award DMS-2023239.
Paul Valiant is supported by NSF award CCF-2127806.

\bibliographystyle{alpha}
\bibliography{references}

\newpage
\appendix

\input{app_score}

\input{app_localMLE}
\input{app_lb}

\end{document}

%% file: intro.tex
\section{Introduction}

We revisit a fundamental problem in statistics: consider a
translation-invariant parametric model
$\{f^{\theta}\}_{\theta \in \R}$ of distributions, where
$f^\theta(x) = f(x-\theta)$.  Suppose there is an arbitrarily chosen
unknown true parameter $\lambda$, and we get i.i.d.~samples from
$f^{\lambda}$.  The task is to accurately estimate $\lambda$ from the
samples.  This problem is known as \emph{location parameter}
estimation in the statistics literature.

Location estimation is a well-studied and general model, including as
a special case the important setting of Gaussian mean estimation.  In
contrast to general mean estimation (where we want to estimate the
mean of a distribution given minimal assumptions such as moment
conditions), in location estimation we are given the shape of the
distribution up to shift.  This advantage lets us handle some
distributions where mean estimation is impossible (e.g., the mean may
not exist), and lets us aim for higher accuracy than is possible
without knowing the distribution.



The classic theory of location estimation is asymptotic,
see~\cite{vanDerVaart:2000asymptotic} for a detailed background.  On
the algorithmic side, it is well-known that the maximum likelihood
estimator (MLE) is asymptotically normal.  Specifically, as the number
of samples $n$ tends to infinity, the distribution of the MLE
converges to a Gaussian centered at the true parameter with variance
$1/(n\I)$, where $\I$ is the \emph{Fisher information} of the
distribution $f$:
\begin{align}\label{eq:fisher}
  \I := \int \frac{(f'(x))^2}{f(x)} \,\d x = \E_{x \sim f}\left[ (\frac{\partial}{\partial x} \log f(x))^2\right].
\end{align}
Conversely, the celebrated Cram\'{e}r-Rao bound states that the
variance of any unbiased estimator must be at least $1/(n\I)$, meaning
that the MLE has mean-squared error that is asymptotically at least as
good as any unbiased estimator.


In the last few decades, motivated by an increasing dependence on data for high-stakes applications, the statistics and computer science communities have shifted focus towards \emph{finite-sample} and \emph{high probability} theories:
1) asymptotic theories assume access to an infinite amount of data, and can in certain cases fail to predict the performance of an algorithm with only a finite number of samples---see the next section for bad examples for the MLE---
2) in high-stakes applications where failure can be catastrophic, it is crucial for predictions to hold except with exponentially small probability.
Yet, classic results bounding the variance or mean-squared error of estimators, such as the Cram\'{e}r-Rao bound, do not readily translate to (tight) high probability bounds.



The goal of this paper is to establish a finite-sample and high
probability theory for the location estimation problem, in both the
algorithmic bound and the estimation lower bound.  Our algorithmic
theory includes a simple yet crucial modification of perturbing
samples by Gaussian noise---corresponding to drawing samples from a
Gaussian-smoothed version of the underlying distribution---before
performing MLE.  We show that this smoothed MLE has finite-$n$
high-probability performance analogous to the Gaussian tail in the
classic asymptotic theory, but replacing the usual Fisher information
with a \emph{smoothed Fisher information}.  The amount of smoothing
required decreases with $n$.


Complementing our upper bound result, we prove a high probability
version of the Cram\'{e}r-Rao bound for Gaussian-smoothed
distributions, showing that for these distributions our sub-Gaussian
accuracy bound, with variance determined by the Fisher information, is
optimal to within a $1+o(1)$ factor.

\subsection{Obstacles to a Finite Sample Theory}
\label{sec:obstacles}

Before discussing our results in detail, we examine two simple distributions where the asymptotic theory predictions for the MLE do not hold in finite samples.
The first example highlights an information-theoretic barrier, that no algorithm can attain the performance predicted by the Gaussian with variance $1/\I$.
The second example, on the other hand, demonstrates how the MLE can be ``tricked" by the distribution, and how an algorithmic remedy is needed to improve its accuracy.


\paragraph{Gaussian with sawtooth noise}
Our first example takes a standard Gaussian, and adds a fine-grained
sawtooth perturbation to it over a bounded region at the center of the
Gaussian, as shown in Figure~\ref{fig:sawtooth}.  The fine-grained
sawtooth has slope either $+\Delta$ or $-\Delta$, alternating over
``teeth'' of width $w$, for $\Delta \gg 1$ and $w \ll 1/\Delta$.  This
sawtooth perturbation barely changes the pdf, but significantly
changes its derivatives, so the Fisher information  $\int (f')^2/f$
grows from $1$ to $\Theta(\Delta^2)$.

Essentially, the sawtooth perturbation makes the distribution easier
to align \emph{within} a tooth, but not \emph{across} teeth. The
asymptotic analysis reflects that: for $n \gg 1/w^2$ where we can
align the teeth correctly, the MLE is much more accurate on the
perturbed distribution.  But for $n \ll 1/w^2$, no algorithm can do
better on the perturbed distribution than on a regular
Gaussian.\footnote{This can be shown by the KL divergence from
  shifting an integer number of teeth of distance about $1/\sqrt{n}$.}
Thus, the normalized estimation accuracy depends on $n$: for large
enough $n$, it has variance $\frac{\Theta(1)}{\Delta^2 n}$, but for
smaller $n$ it has variance $\frac{1}{n}$.  Our finite sample theory
should reflect this.

\begin{figure}
    \begin{subfigure}[b]{.65\textwidth}
        \centering
        \hspace*{-1.5cm}
        \includegraphics[trim={2.5cm 0cm .5cm .5cm},clip,valign=m,width=.25\textwidth]{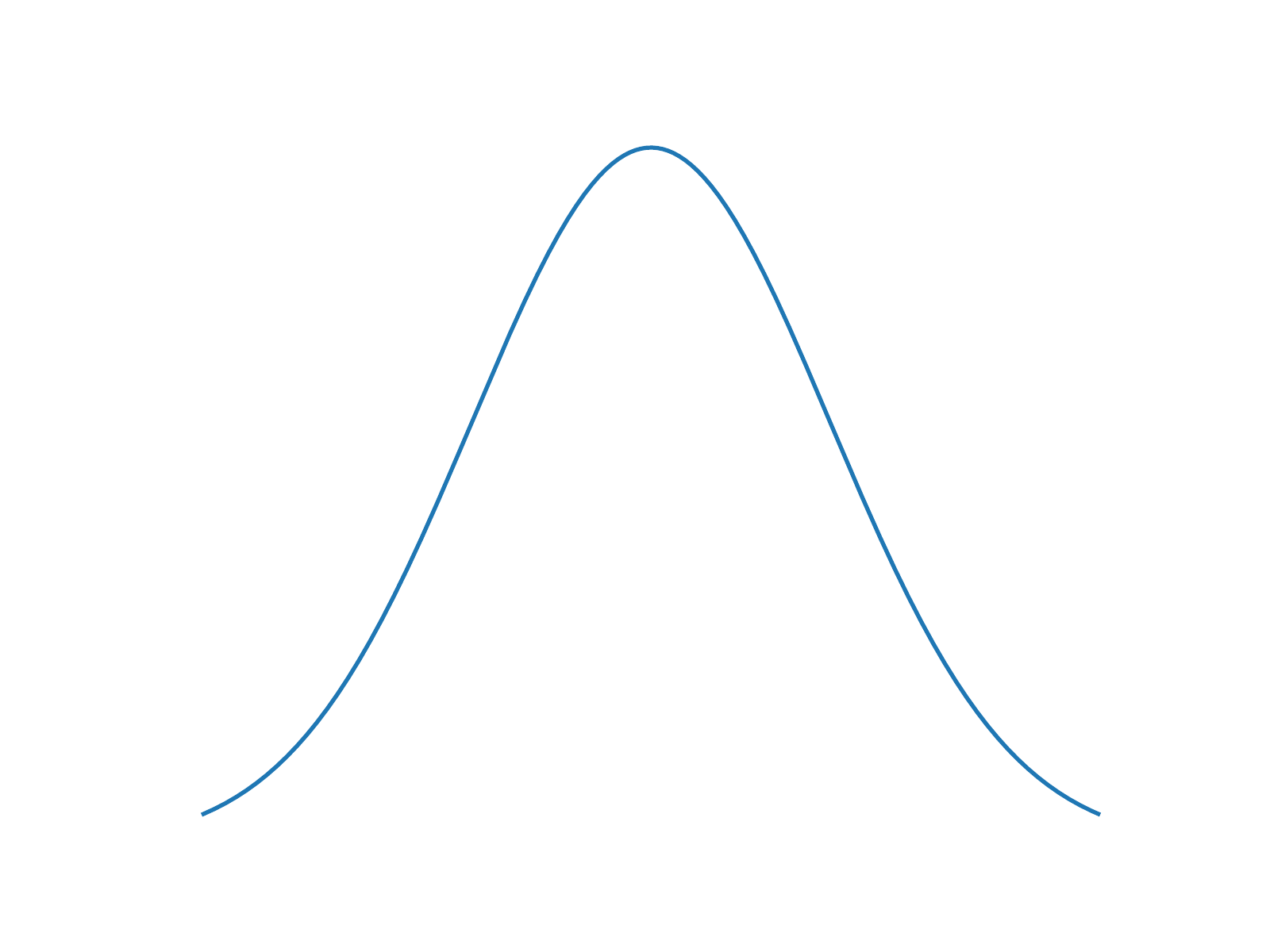}\hspace*{-0.5cm}
        \Huge$+$\hspace*{0cm}
        \includegraphics[trim={2.5cm 0cm .5cm .5cm},clip,valign=m,width=.25\textwidth]{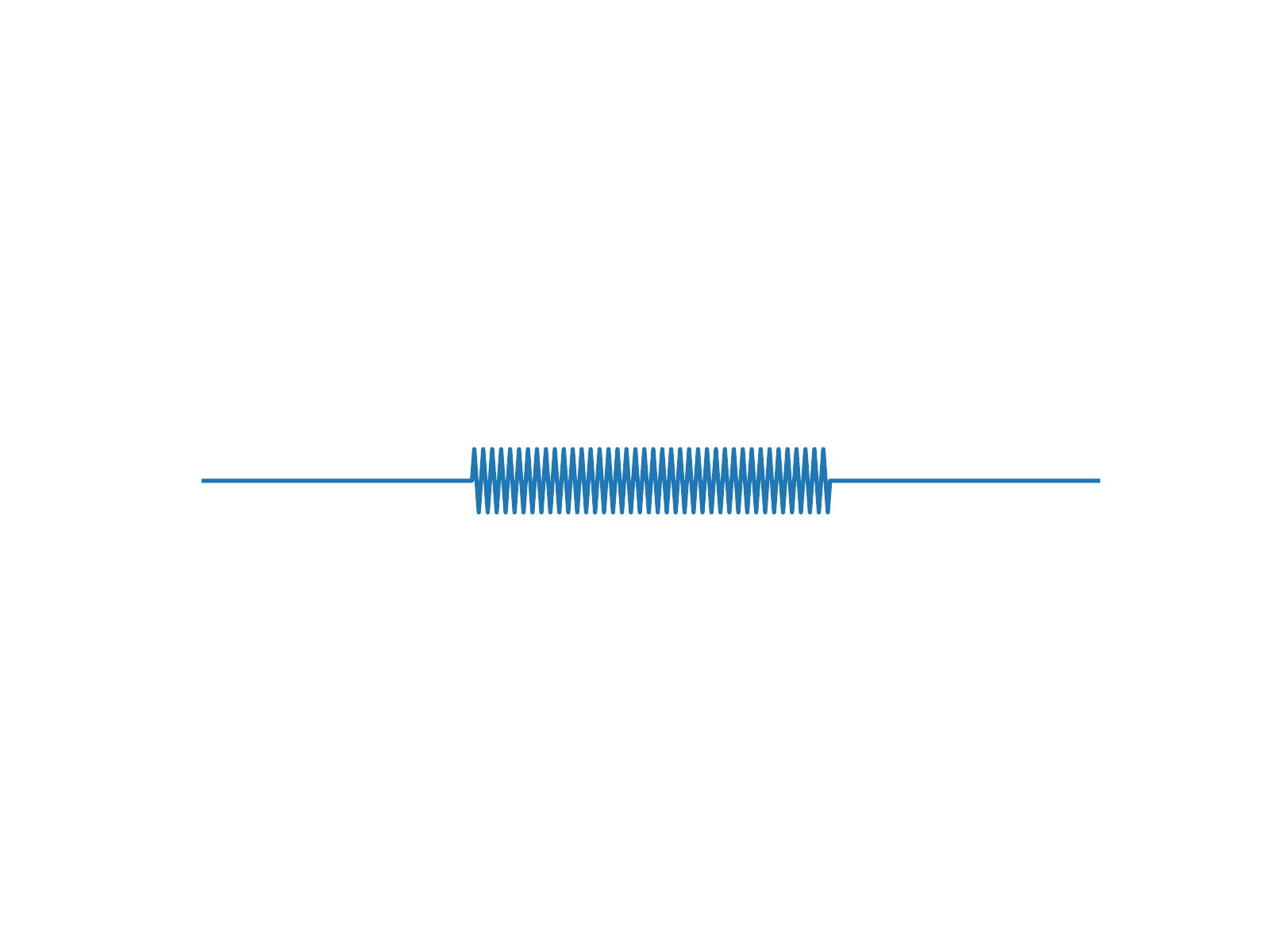}
        \Huge$=$
        \includegraphics[trim=3.5cm 0cm 3.5cm 0cm,clip,valign=m,width=.25\textwidth]{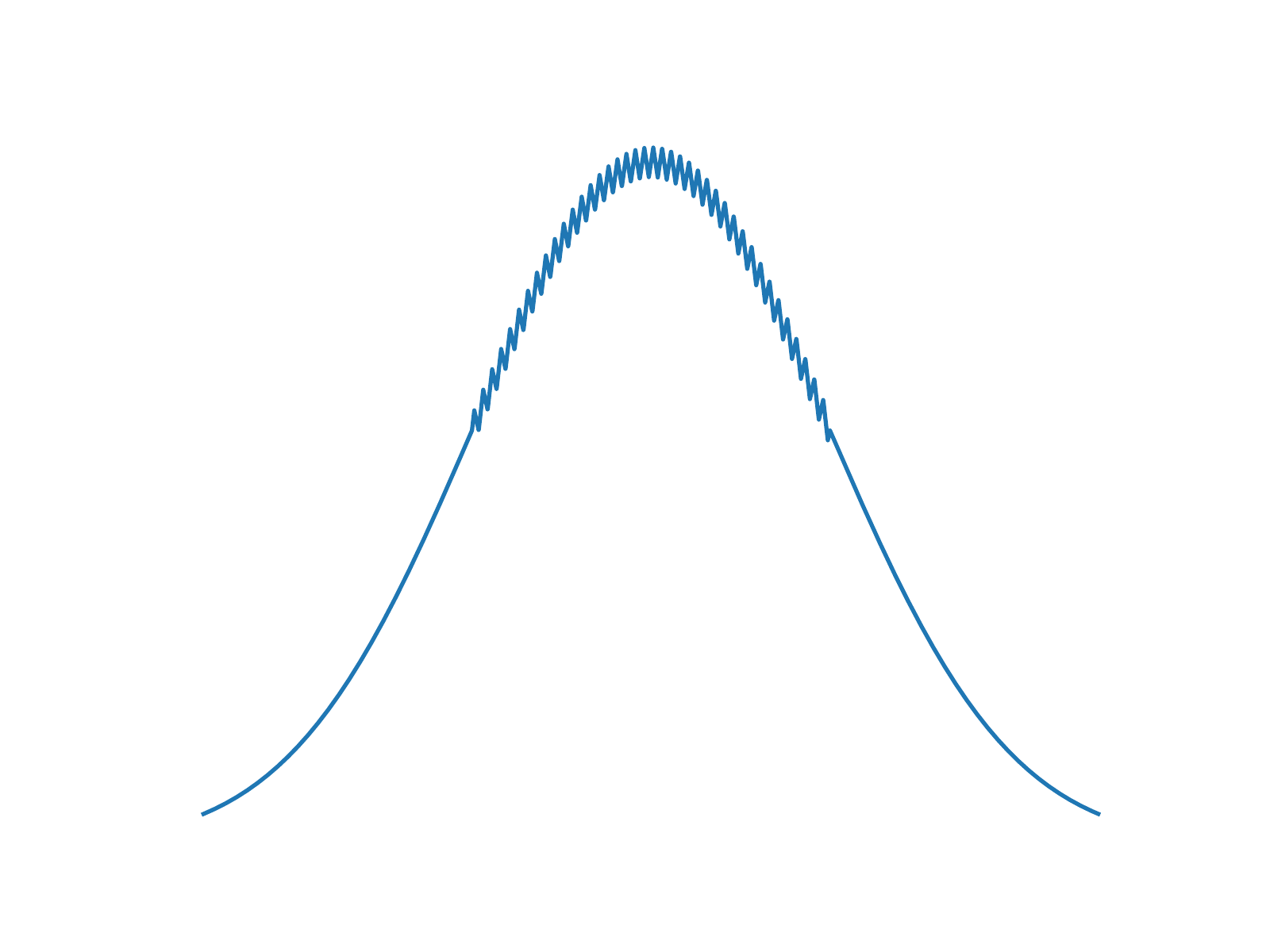}
        \caption{Sawtooth example}
        \label{fig:sawtooth}
    \end{subfigure}
    \hfill
    \begin{subfigure}[b]{.35\textwidth}
    \centering
        \includegraphics[width=\textwidth,keepaspectratio]{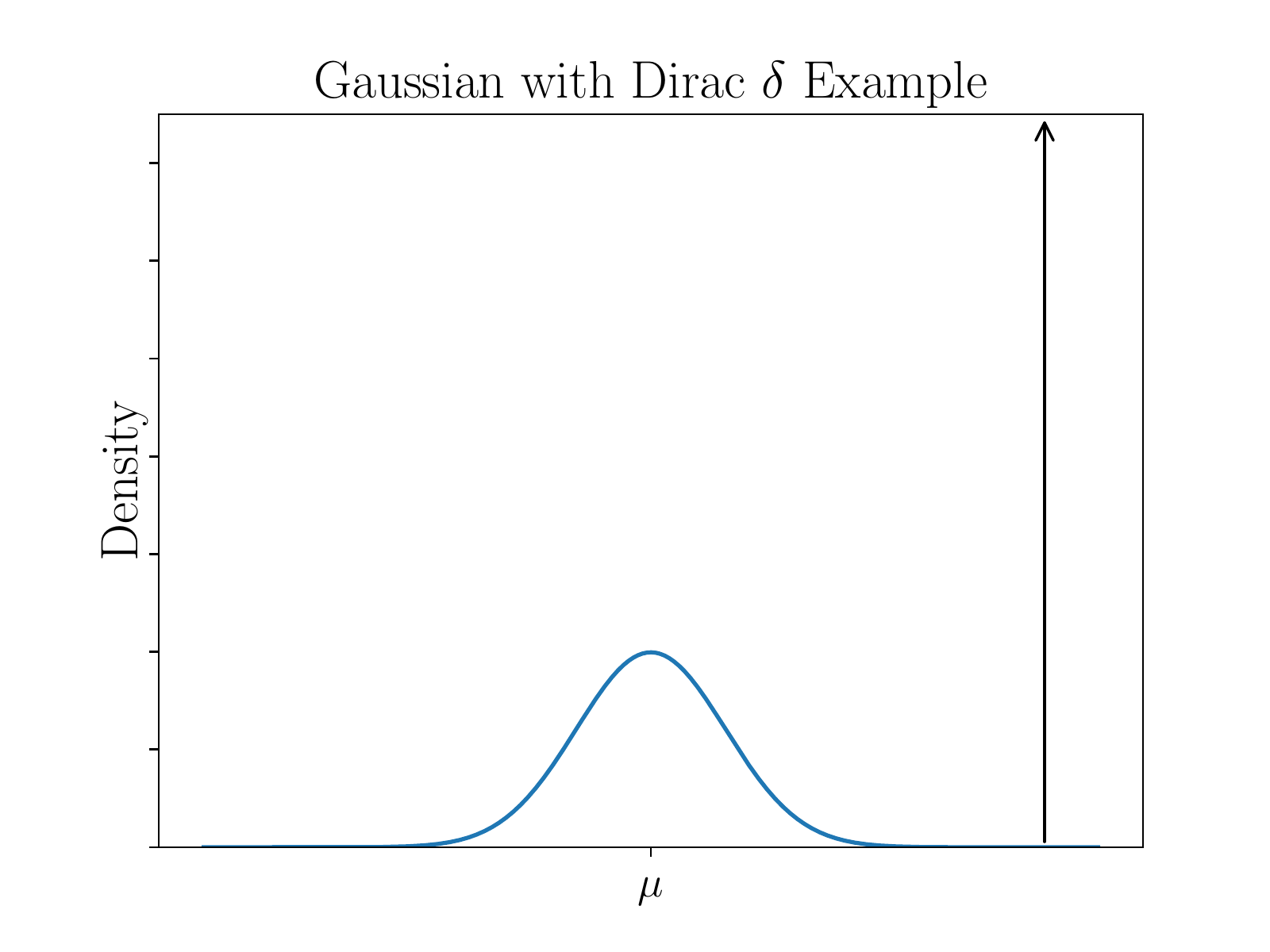}
        \caption{Dirac $\delta$ example}
        \label{fig:spikedGaussian}
    \end{subfigure}
    \caption{Bad examples for MLE}
    \vspace*{-5mm}
\end{figure}


\paragraph{Gaussian with a Dirac $\delta$ spike}
Our next example adds an even simpler noise to the standard Gaussian:
a Dirac $\delta$ spike with minimal mass $\eps$, placed sufficiently
far away from the Gaussian mean 0 (Figure~\ref{fig:spikedGaussian}).
This distribution has infinite Fisher information, reflecting that for
$n \gg 1/\eps$, we will probably see the spike multiple times, identify
it, and get zero estimation error.  But for $n < 1/\eps$, we probably
will not see the spike, so our error bound will not reflect the
overall Fisher information.

Moreover, the MLE performs remarkably badly when $n \ll 1/\eps$.  Most
likely the Dirac $\delta$ is not sampled, yet the Dirac $\delta$ has
infinite density, and so the MLE will match a Gaussian sample to the
spike to get a maximium likelihood solution.  As the spike is placed
far from the true mean, this leads to much higher error than (say) the
empirical mean.

Fortunately, there is a simple modification to the MLE that solves
this issue almost completely: we add a small amount of Gaussian noise
to the distribution.  This means we convolve the PDF with a small
Gaussian, and add independent Gaussian noise to each individual sample
to effectively draw from the convolved PDF.  The crucial effect of this
Gaussian smoothing is that the density of the Dirac $\delta$ will be
reduced from infinity down to some constant, and hence MLE will no
longer ``fit it at any cost''.  On the other hand, the smoothing increases
the variance of the distribution and decreases the Fisher information.
This raises the question of determining the \emph{best} amount of
smoothing, which we address in this paper.


\subsection{Our Results and Approach}
\label{sec:results}

Based on the discussion of the spiked Gaussian model in the previous section, we propose a finite sample theory for MLE that uses Gaussian smoothing.
As we described, the algorithmic approach is simple: we perturb all the samples by independent Gaussian noise of variance $r^2$, and convolve the known model $f$ by the same Gaussian to yield the model $f_r$, before performing MLE.


We state below the basic MLE algorithm (Algorithm~\ref{alg:localMLE}) in this paper, which adopts the above approach.
Algorithm~\ref{alg:localMLE} is a \emph{local} algorithm, in that it assumes as input an initial uncertainty region guaranteed to contain the true parameter $\lambda$, and performs MLE only over this domain.
Furthermore, Algorithm~\ref{alg:localMLE} attempts only to find a local optimum in the likelihood: it computes the derivative of the log likelihood function, also known as the \emph{score} function, and returns any root of the score function.

\begin{algorithm}\caption{Local MLE for known parametric model}
\label{alg:localMLE}
\vspace*{-3mm}
\paragraph{Input Parameters:} Description of distribution $f$, smoothing parameter $r$, samples $x_1, \ldots, x_n \overset{i.i.d.}{\sim} f^{\lambda}$, uncertainty region $[\ell, u]$ containing the unknown $\lambda$
\begin{enumerate}
\item Let $s_r(\lambdahat)$ be the score function of $f_r$, the $r$-smoothed version of $f$.
\item For each sample $x_i$, compute a perturbed sample $x'_i = x_i + \Normal(0,r^2)$ where all the Gaussian noise are drawn independently across all the samples.
\item Compute $\lambdahat$ that is a root of the empirical score function $\hat{s}(\lambdahat) = \sum_{i = 1}^n s_r(x'_i-\lambdahat)$ inside the domain $[\ell,u]$.
A root should exist and picking any root is sufficient.
\item Return $\lambdahat$.
\end{enumerate}

\end{algorithm}

Theorem~\ref{thm:simpleLocalMLE} states our guarantees on Algorithm~\ref{alg:localMLE}.
Locally around the true parameter---that is, within $r/2$ for the $r$-smoothed distribution---any root of the score function (i.e., local optimum of likelihood) must be very close to the true parameter $\lambda$ with high probability over the $n$ samples.  In particular, the estimation error is within a $1 + o(1)$ factor of the Gaussian deviation with variance $\frac{1}{n\I_r}$ and failure probability $\delta$, where $\I_r$ is the Fisher information of $f_r$.



\begin{restatable}[Local Convergence]{theorem}{corSimpleLocalMLE}
\label{thm:simpleLocalMLE}
Suppose we have a known model $f_r$ that is the result of $r$-smoothing, with Fisher information $\I_r$, and a given width parameter $\eps_{\max}$ and failure probability $\delta$.
Further suppose that $r$ satisfies $r \ge 2\eps_{\max}$ and there is a sufficiently large parameter $\gamma$ such that 
1) $r^2 \sqrt{\I_r} \ge \gamma \eps_{\max}$,
2) $(\log \frac{1}{\delta})/n \le \frac{1}{\gamma^2}$ and
3) $\log 1/(r\sqrt{\I_r}) \le \frac{1}{\gamma}\log\frac{1}{\delta}/\log \log \frac{1}{\delta}$.

Then, with probability at least $1-\delta$, for all $\eps \in \left((1+O(\frac{1}{\gamma}))\sqrt{\frac{2\log\frac{1}{\delta}}{n \I_r}}, \eps_{\max}\right]$, $\hat{s}(\lambda-\eps)$ is strictly negative and $\hat{s}(\lambda+\eps)$ is strictly positive.
\end{restatable}

In the scenario where we do have an initial uncertainty region for the true parameter $\lambda$, we would use Theorem~\ref{thm:simpleLocalMLE} to compute the minimal smoothing amount $r$ satisfying the assumptions in the theorem, then use Algorithm~\ref{alg:localMLE} with this parameter $r$, to obtain an accurate estimate of $\lambda$.

\looseness=-1 In general, however, we may not have an initial uncertainty region for $\lambda$.
In Section~\ref{sec:globalMLE}, we present Algorithm~\ref{alg:globalMLE}, a global two-stage MLE algorithm, which first infers an initial uncertainty region by using quantile information from the distribution $f$ before invoking Algorithm~\ref{alg:localMLE}, the local MLE algorithm.
The guarantees of Algorithm~\ref{alg:globalMLE} are summarized here.

\begin{theorem}[Global MLE guarantees, informal version of Theorem~\ref{thm:simpleGlobalMLE}]
\label{thm:informalSimpleGlobalMLE}
Given a model $f$, let the $r$-smoothed Fisher information of a
distribution $f$ be $\I_r$, and let $\IQR$ be the interquartile range
of $f$.  When
$n \gg \log\frac{1}{\delta} \gg 1$, there exists an $r^* = o(\IQR)$ such
that, with probability at least $1-\delta$, the output $\lambdahat$ of
Algorithm~\ref{alg:globalMLE} satisfies
\[ |\lambdahat - \lambda| \le \left(1+o(1)\right)\sqrt{\frac{2\log\frac{1}{\delta}}{n I_{r^*}}}\]
\end{theorem}

In addition to the theoretical framework, Section~\ref{sec:experiments} gives experimental evidence demonstrating that $r$-smoothed Fisher information does capture the empirical performance of (smoothed) MLE.

We also prove new estimation lower bounds for the location estimation problem for $r$-smoothed distributions.
The lower bound statement below (Theorem~\ref{thm:lowerbound}) shows that the estimation error $\left(1+o(1)\right)\sqrt{\frac{2\log\frac{1}{\delta}}{n I_{r}}}$ is optimal to within a $1+o(1)$ factor.

\begin{restatable}{theorem}{ThmLB}
\label{thm:lowerbound}
Suppose $f_r$ is an $r$-smoothed distribution with Fisher information $\I_r$.
Given failure probability $\delta$ and sample size $n$, no algorithm can distinguish $f_r$ and $f_r^{2\eps}$ with probability $1-\delta$, where $\eps = (1-o(1))\sqrt{2\log\frac{1}{\delta}/(n\I_r)}$.
Here, the $o(1)$ term tends to 0 as $\delta \to 0$ and $\log\frac{1}{\delta}/n \to 0$, for a fixed $r^2\I_r$.
\end{restatable}

This lower bound is the standard ``two-point" statement that, with $n$ samples, it is statistically impossible to distinguish between the distributions $f$ and $f$ shifted by a small error (in the $x$-axis) with probability $1-\delta$.
Even though there are known standard inequalities on distribution distances and divergences for proving lower bounds of this form, the technical challenge is that they generally yield estimation lower bounds that are only tight to within constant factors, instead of the $1+o(1)$ tightness we desire.
This paper presents new analysis to derive a $1+o(1)$-tight lower bound, which may be of independent interest.

\subsection{Notations}

We denote the shift-invariant model we consider by the distribution $f$, and the distribution with parameter $\lambda$ by $f^\lambda(x) = f(x-\lambda)$.
Denote by $Z_r$ the Gaussian with mean 0 and variance $r^2$.
The $r$-smoothed model for $f$ is denoted by $f_r$ (and similarly, for parameter $\lambda$, $f^\lambda_r$) which is distributed as $Y = Z_r + X$ where $X \from f$ independently from the Gaussian perturbation $Z_r$.

\looseness=-1 The log-likelihood function of $f$ is denoted by $l = \log f$.
The \emph{score} function is the derivative of $l$, denoted by $s = l' = f'/f$.
We use the notation $s_r$ to denote the score function of $f_r$.
The Fisher Information of $f$ is denoted by $I = \Exp_{x \from f}[s^2(x)]$.
Similarly, the Fisher Information of $f_r$ is denoted by $I_r$.

%% file: literature.tex
\section{Related Work}

Location estimation and MLE in general has been extensively studied under the lens of asymptotic statistics.
See~\cite{vanDerVaart:2000asymptotic} for an in-depth treatment.
The MLE has also been studied under the finite-sample setting~\cite{parametric_finite_sample,rate_of_convergence_mle, MLE_concentration}, but these prior works impose restrictive regularity conditions and also loses (at least) multiplicative constants in the estimation accuracy.
In contrast, our work modifies the MLE to include smoothing, and we give analyses that are tight to within $1+o(1)$ factors.

\looseness=-1 There has also been a flurry of recent interest in the related mean estimation problem, in the finite-sample and high-probability setting.
Recall that mean estimation does not assume knowledge of the shape of the distribution, but instead imposes mild moment conditions, for example the finiteness of the variance.
Catoni~\cite{catoni} initiated a line of work studying the statistical limits of univariate mean estimation to within a $1+o(1)$ factor, ending recently with the work of Lee and Valiant~\cite{lee_valiant}, which proposed and analyzed an estimator with accuracy optimal to within a $1+o(1)$ factor for all distributions with finite variance.
See also the recent work of Minsker~\cite{Minsker:2022-subgaussian-mean} for an alternative solution.

Beyond the differences in assumptions, the main distinction between location and mean estimation lies in their statistical limits.
In mean estimation, the optimal accuracy is captured by the variance of the underlying distribution, scaling linearly with the standard deviation.
On the other hand, the classic asymptotic theory suggests that the Fisher information captures the optimal accuracy for location estimation, scaling with the reciprocal of the square root of the Fisher information.
It is a well-known fact that the Fisher information is always lower bounded by the reciprocal of the variance~\cite{Shevlyakov:2011}, which shows that location estimation is always easier than mean estimation in the infinite-sample regime.
In this work, we refine this understanding, showing that in finite samples, the optimal accuracy for location estimation is instead given by the $r$-smoothed Fisher information in place of the unsmoothed Fisher information.



%% file: score.tex
\section{Tails and boundedness of $r$-smoothed score and Fisher information}
\label{sec:rscore}

Recall that given a distribution $f$, its $r$-smoothed version $f_r$ is distributed as $Y = X+Z_r$ where $X \sim f$ and $Z_r \sim \Normal(0, r^2)$ and $X,Z_r$ are independent.

Both our algorithmic and lower bound theories are centered around $r$-smoothed distributions.
Therefore, we state here basic concentration and boundedness properties of $r$-smoothed score function and Fisher information, which we use in the rest of the paper.
We prove all these lemmas in Appendix~\ref{app:rscore}.




First, we show that the $r$-smoothed Fisher information $\I_r$ is upper bounded by $1/r^2$ and can be lower bounded using the interquartile range of $f$.

\begin{restatable}{lemma}{LemIrBounded}
\label{lem:IrBounded}
  Let $\I_r$ be the Fisher information of an $r$-smoothed distribution $f_r$.
  Then, $\I_r \le 1/r^2$.
\end{restatable}

\begin{restatable}{lemma}{LemIrLB}
\label{lem:IrLB}
Let $\I_r$ be the Fisher information for $f_r$, the $r$-smoothed version of distribution $f$.
Let $\IQR$ be the interquartile range of $f$.
Then, $\I_r \gtrsim 1/(\IQR+r)^2$.
Here, the hidden constant is a universal one independent of the distribution $f$ and independent of $r$.
\end{restatable}

Next, we show that, fixing a point close to the true parameter $\lambda$, the empirical score function evaluated at that point will concentrate around its expectation for smoothed distributions.

\begin{restatable}{corollary}{CorScoreEstimate}
\label{cor:scoreestimate}
Let $f$ be an arbitrary distribution and let $f_r$ be the $r$-smoothed version of $f$.
That is, $f_r(x) = \E_{y \from f}[\frac{1}{\sqrt{2\pi r^2}}e^{-\frac{(x-y)^2}{2r^2}}]$.
Consider the parametric family of distributions $f_{r}^{\lambda}(x) = f_r(x-\lambda)$.
Suppose we take $n$ i.i.d.~samples $y_1,\ldots,y_n \from f_r^\lambda$, and consider the empirical score function $\hat{s}$ mapping a candidate parameter $\hat{\lambda}$ to $\frac{1}{n}\sum_i s_r(y_i-\lambdahat)$, where $s_r$ is the score function of $f_r$.

Then, for any $|\eps| \le r/2$,

\[
  \Pr_{y_i \overset{i.i.d.}{\sim} f_r^\lambda}\left(|\hat{s}(\lambda+\eps) - \E_{x\from f_r}[s(x-\eps)]| \ge \sqrt{\frac{2\max(\E_x[s^2_r(x - \eps)], \I_r)\log\frac{2}{\delta}}{n}} + \frac{15\log\frac{2}{\delta}}{nr}\right) \le \delta
\]
\end{restatable}

%% file: localMLE.tex
\section{A Finite Sample Analysis of $r$-smoothed Local MLE}
\label{sec:localMLE}

In this section, we analyze Algorithm~\ref{alg:localMLE}, which is our version of local MLE with $r$-smoothing applied.
Algorithm~\ref{alg:localMLE} takes an initial uncertainty region that the true parameter is guaranteed to lie in, and uses the model and the initial interval to refine the estimate to high accuracy.
We first present a simpler and easier-to-interpret version of our result, Theorem~\ref{thm:simpleLocalMLE}, which we stated in Section~\ref{sec:results}.

Recall that Algorithm~\ref{alg:localMLE} computes the empirical score function, and returns any of its roots.
The theorem thus states that, with high probability, for any point $\lambda+\eps$ with $|\eps|$ too large, the empirical score function must be non-zero and thus $\lambda+\eps$ will not returned as the estimate.
More precisely, given an initial interval of length $\eps_{\max}$ as well as the failure probability $\delta$, the theorem assumes that the smoothing parameter $r$ is sufficiently large (conditions 1 and 3 in the theorem) and that the sample size $n$ is sufficiently large, and guarantees an estimation error that is within a $1+o(1)$ factor of the error predicted by the Gaussian with variance $1/\I_r$, where $\I_r$ is the Fisher information of $f_r$.

\corSimpleLocalMLE*

The above theorem follows from the following theorem, which makes the ``$o(1)$" term (the $O(1/\gamma)$ term) in the theorem explicit.
Assumptions 2 and 3 in the theorem statement essentially bounds various multiplicative terms in the estimation error and makes sure that they are ``$1+o(1)$" terms.

\begin{theorem}
\label{thm:localMLE}
Suppose we have a known model $f_r$ that is the result of $r$-smoothing, and a given parameter $\eps_{\max}$.
Let $\beta$ and $\eta$ be the hidden multiplicative constants in Lemmas~\ref{lem:grad_expectation_theta} and~\ref{lem:var_close_to_fisher}.
Further suppose that $r$ satisfies $r \ge 2\eps_{\max}$ and $r^2 \sqrt{\I_r} \ge \gamma \eps_{\max}$ for some parameter $\gamma \ge \beta$.

Now define the notation $\rho_r$ by 
\[ 1+\rho_r = \sqrt{1+\frac{\eta}{\gamma}} + \frac{15}{2\sqrt{\gamma}}\left(\frac{2\log\frac{4\log\frac{1}{\delta}}{r^2 \I_r(1-\frac{\beta}{\gamma})\delta}}{n}\right)^{\frac{1}{4}} \]
Then, for sufficiently small $\delta > 0$, with probability at least $1-\delta$, for all $\eps \in \left((1+\frac{1}{\log\frac{1}{\delta}})\frac{1+\rho_r}{1-\frac{\beta}{\gamma}}\sqrt{1+\frac{\log\frac{4\log\frac{1}{\delta}}{r^2 \I_r (1-\frac{\beta}{\gamma})}}{\log\frac{1}{\delta}}} \sqrt{\frac{2\log\frac{1}{\delta}}{n \I_r}}, \eps_{\max}\right]$, $\hat{s}(\lambda-\eps) < 0$ and $\hat{s}(\lambda+\eps) > 0$.
\end{theorem}

To prove Theorem~\ref{thm:localMLE}, it suffices to show the following lemma.
The theorem follows directly by reparameterizing $\delta$ and choosing $\xi$ to be $1/\log\frac{1}{\delta}$.

\begin{restatable}{lemma}{LemLocalMLE}
\label{lem:localMLE}
Suppose we have a known model $f_r$ that is the result of $r$-smoothing with Fisher information $\I_r$, and a given parameter $\eps_{\max}$.
Let $\beta$ and $\eta$ be the hidden multiplicative constants in Lemmas~\ref{lem:grad_expectation_theta} and~\ref{lem:var_close_to_fisher}.
Further suppose that $r$ satisfies $r \ge 2\eps_{\max}$ and $r^2 \sqrt{\I_r} \ge \gamma \eps_{\max}$ for some parameter $\gamma \ge \beta$.
Also define the notation $\tilde{\rho}$ (a ``$o(1)$" term) by
\[ 1+\tilde{\rho} = \sqrt{1+\frac{\eta}{\gamma}} + \frac{15}{2\sqrt{\gamma}}\left(\frac{2\log\frac{1}{\delta}}{n}\right)^{\frac{1}{4}} \]

Then, for every $\xi \ll 1$, with probability at least $1-\delta \cdot \frac{2}{\xi r^2\I_r(1-\frac{\beta}{\gamma})(1-\delta)}$, for all $\eps \in \left((1+\xi)\frac{1+\tilde{\rho}}{1-\frac{\beta}{\gamma}} \sqrt{\frac{2\log\frac{1}{\delta}}{n \I_r}}, \eps_{\max}\right]$, $\hat{s}(\lambda-\eps)$ is strictly negative and $\hat{s}(\lambda+\eps)$ is strictly positive.

\end{restatable}

We prove Lemma~\ref{lem:localMLE} in Appendix~\ref{app:localMLE}, and here we give a proof sketch.

\begin{proof}[Proof sketch for Lemma~\ref{lem:localMLE}]
First, recall that Corollary~\ref{cor:scoreestimate} from Section~\ref{sec:rscore} shows that fixing a candidate input value $\lambda+\eps$ for some small $\eps$, the value of the empirical score function at $\lambda+\eps$ is well-concentrated around its expectation.
In Lemmas~\ref{lem:grad_expectation_theta} and~\ref{lem:var_close_to_fisher}, we calculate and bound the expectation and second moment of the empirical score function at $\lambda+\eps$ for all sufficiently small $\eps$.
This allows us to derive tail bounds for the empirical score function at each point $\lambda+\eps$, to show that it is bounded away from 0.
Next, we need to show that with high probability, the empirical score function is \emph{simultaneously} bounded away from 0 for all $\eps$ with magnitude greater than the desired estimation accuracy.
We achieve this via a straightforward net argument, crucially utilizing the fact that the expectation of the empirical score function is bounded away from 0 by an essentially linear function in $\eps$, and that the variance is essentially constant in $\eps$.
This means that the probability for the empirical score function at $\lambda+\eps$ to hit 0 is decreasing exponentially in $\eps$, which allows us to complete the net argument.
\end{proof}


%% file: globalMLE.tex
\section{Global Two-Stage MLE Algorithm}
\label{sec:globalMLE}

Algorithm~\ref{alg:localMLE}, which we stated in the introduction and analyzed in Section~\ref{sec:localMLE}, is a \emph{local} algorithm that assumes we have knowledge of a non-trivially small uncertainty region containing the true parameter $\lambda$.
The smoothing parameter $r$ can then be computed from the assumptions of Theorems~\ref{thm:localMLE} or~\ref{thm:simpleLocalMLE}, and we run Algorithm~\ref{alg:localMLE} to obtain an accurate estimate of the true parameter $\lambda$, with accuracy predicted by the $r$-smoothed Fisher information $\I_r$.

However, in general, we might not have a-priori knowledge of where the true parameter $\lambda$ lies.
In this section, we propose a \emph{global} maximum likelihood algorithm (Algorithm~\ref{alg:globalMLE}) which first estimates a preliminary interval containing $\lambda$, before choosing the smoothing parameter $r^*$ using an \emph{easily calculable} expression that is $o(1)$ times smaller than the interquartile range of the distribution, and finally applies the local MLE algorithm (Algorithm~\ref{alg:localMLE}) to obtain a final estimate.
Theorem~\ref{thm:simpleGlobalMLE} states that the accuracy of Algorithm~\ref{alg:globalMLE} is always within a $1+o(1)$ times the accuracy predicted by the $r^*$-smoothed Fisher information $\I_{r^*}$.

\begin{algorithm}
\caption{Global MLE for known parametric model}
\label{alg:globalMLE}
\vspace*{-1.5mm}
\paragraph{Input Parameters:} Failure probability $\delta$, description of distribution $f$, $n$ i.i.d.~samples drawn from $f^\lambda$ for some unknown $\lambda$
\begin{enumerate}
\item Compute an $\alpha \in [\sqrt{2\log\frac{4}{\delta}/n}, 1-\sqrt{2\log\frac{4}{\delta}/n}]$ such that the interval defined by the $\alpha-\sqrt{2\log\frac{4}{\delta}/n}$ and $\alpha+\sqrt{2\log\frac{4}{\delta}/n}$ quantiles of $f$ is the smallest.

\item By standard Chernoff bounds, with probability at least $1-\frac{\delta}{2}$, the sample $\alpha$-quantile $x_\alpha$ will be such that $x_\alpha - \lambda$ is within the $\alpha-\sqrt{2\log\frac{4}{\delta}/n}$ and $\alpha+\sqrt{2\log\frac{4}{\delta}/n}$ quantiles of $f$.
Based on this, compute an initial confidence interval $[\ell, u]$ for $\lambda$.


\item Let $r^* = \Omega(\max((\frac{\log \frac{1}{\delta}}{n})^{1/8},2^{-O(\sqrt{\log\frac{1}{\delta}})})) \IQR$.
\item Run Algorithm~\ref{alg:localMLE} on the interval $[\ell, u]$, using $r^*$-smoothing and failure probability $\delta/2$, returning the final estimate $\lambdahat$.
\end{enumerate}

\end{algorithm}

\begin{theorem}[Global MLE Theorem]
\label{thm:simpleGlobalMLE}
Given a model $f$, let the $r$-smoothed Fisher information of a distribution $f$ be $\I_r$, and let $\IQR$ be the interquartile range of $f$.
Fix the failure probability be $\delta$.

Choose $r^* =  \Omega(\max((\frac{\log \frac{1}{\delta}}{n})^{1/8},2^{-O(\sqrt{\log\frac{1}{\delta}})})) \IQR$.
Then, with probability at least $1-\delta$, the output $\lambdahat$ of Algorithm~\ref{alg:globalMLE} satisfies
\[ |\lambdahat - \lambda| \le \left(1+O\left(\frac{\log \frac{1}{\delta}}{n}\right)^{\frac{1}{4}} + O\left(\frac{1}{\sqrt{\log\frac{1}{\delta}}}\right) \right)\sqrt{\frac{2\log\frac{1}{\delta}}{n \I_{r^*}}}\]
\end{theorem}

\begin{proof}
  The total failure probability of the steps is at most $\delta$.
  Thus, in this proof we condition on the success of Algorithm~\ref{alg:localMLE} in all probabilistic steps.

    By the minimality condition in the definition of $\alpha$, the length $\eps_{\max}$ of the interval $[\ell, u]$ from Step 2 is at most $O(\sqrt{\log\frac{1}{\delta}/n})\IQR$.
    
    Further, recall by Lemma~\ref{lem:IrLB} that $\I_r \ge \Omega(1/(\IQR+r)^2)$.
    Picking $r^* = \Omega(\max((\frac{\log \frac{1}{\delta}}{n})^{1/8},2^{-O(\sqrt{\log\frac{1}{\delta}})})) \IQR$ and $\gamma_1 = O(\frac{n}{\log \frac{1}{\delta}})^{1/4}$ and $\gamma_2 = O(\sqrt{\log\frac{1}{\delta}})$, we check that the following conditions are satisfied:
    \begin{enumerate}
        \item $(r^*)^2\sqrt{\I_{r^*}} \ge (r^*)^2/(\IQR+r^*) \ge \Omega(\frac{\log \frac{1}{\delta}}{n})^{1/4} \IQR = \gamma_1 \eps_{\max}$.
        \item $\log\frac{1}{\delta}/n \le O(\sqrt{\log\frac{1}{\delta}/n}) \le 1/\gamma_1^2$.
        \item $\log 1/(r^*\sqrt{\I_{r^*}}) \le O(\log 2^{O(\sqrt{\log\frac{1}{\delta}})}) = O(\frac{1}{\gamma_2}\log\frac{1}{\delta})$.
    \end{enumerate}
    Further note that $\log\log\frac{1}{\delta}/\log\frac{1}{\delta} \le 1/\sqrt{\log\frac{1}{\delta}} = O(1/\gamma_2)$.
    
    Thus, using Theorem~\ref{thm:simpleLocalMLE}, Step 4 returns an estimate $\lambdahat$ satisfying
    \[ |\lambdahat - \lambda| \le \left(1+O\left(\frac{1}{\gamma_1}\right) + O\left(\frac{1}{\gamma_2}\right) \right)\sqrt{\frac{2\log\frac{1}{\delta}}{n \I_{r^*}}}\]
    which is equivalent to the theorem statement.
\end{proof}

%% file: lowerbound.tex
\section{High Probability Cram\'{e}r-Rao Bound}
\label{sec:lb}

Complementing our algorithmic results, we also give new results on \emph{lower bounding} the estimation error in the location parameter model.
The celebrated Cram\'{e}r-Rao bound lower bounds the variance of estimators, which does not readily translate to (tight) lower bounds on the distribution tail of the estimation error.
In this section, we show that it is possible to derive a high probability version of the Cram\'{e}r-Rao lower bound for $r$-smoothed distributions, where, given a failure probability $\delta$, we lower bound the estimation error to within a $1+o(1)$-factor of the error predicted by the asymptotic normality of the standard maximum likelihood algorithm, namely the Gaussian with the true parameter as the mean, and variance $1/(n\I_r)$ for estimation using $n$ samples.


\ThmLB*

We prove Theorem~\ref{thm:lowerbound} in Appendix~\ref{app:lb}.
The high-level technique we use is a standard one, showing that, it is statistically impossible to distinguish two slightly shifted copies of $f_r$ with probability $1-\delta$, using $n$ samples.
The shift corresponds to (twice) the estimation accuracy lower bound.  The difficulty lies in getting the right constant.

Standard inequalities for showing indistinguishability results rely on calculating either the squared Hellinger distance~\cite{bar-yossef:2002complexity} or the KL-divergence~\cite{bretagnolle:1979estimation} between the two distributions.
While these inequalities are straightforward to apply, given the calculated bounds on these statistical distances/divergences, the inequalities only yield constant-factor tightness in the estimation accuracy lower bound.
On the other hand, in this work, we aim to give accuracy upper and lower bounds that are matching strongly, to within $1+o(1)$ factors.
As such, our proof of Theorem~\ref{thm:lowerbound} involves delicate and non-standard bounding techniques which may be of independent interest.
The proof techniques are currently slightly ad-hoc, and for future work, we hope to improve on these techniques to make them more general and more usable.

%% file: experiments.tex
\section{Experimental Results}
\label{sec:experiments}

In this section, we give experimental evidence supporting our proposed algorithmic theory.
Our goals are to demonstrate that 1) $r$-smoothing is a beneficial pre-processing to the MLE, that $r$-smoothed Fisher information does capture the algorithmic performance in location estimation and 2) $r$-smoothed MLE can outperform the standard MLE, as well as standard mean estimation algorithms which do not leverage information about the distribution shape.

The version of smoothed MLE we use for experimentation is even simpler than Algorithm~\ref{alg:localMLE}: use Gaussian smoothing before performing actual maximum likelihood finding over the entire real line, instead of returning a root of the empirical score function.
This is closest to what statisticians would do in practice, and further does not require any initial uncertainty region on the true parameter $\lambda$.

\looseness=-1 We use the Gaussian-spiked \emph{Laplace} model for experiments, with a Laplace distribution of density proportional to $e^{-|x|}$, and a Gaussian of mass 0.001 and width roughly 0.002 (the discretization granularity) added at $x=4$.
The reason we choose the Laplace over the Gaussian as the ``body" of the distribution is because, fixing the variance of the distribution, the Laplace has twice the Fisher information as the Gaussian.
Given that standard mean estimation algorithms only aim to achieve sub-Gaussian concentration, choosing the Laplace as the core distribution lets us demonstrate that the smoothed MLE can outperform mean estimation algorithms even in finite samples.
We also note that this example is crucially different from the Dirac $\delta$-spiked example in Section~\ref{sec:obstacles}.
The Dirac $\delta$ spike has infinity density, whereas the narrow Gaussian spike only has somewhat large, but finite density.
Given the finite and not too large density in the spike, in our experiments, the basic MLE algorithm will \emph{not} ``fit it at any cost", and instead has a smoother error distribution whenever we do not observe samples from the Gaussian spike.
Nonetheless, even in this milder setting, we demonstrate that our smoothed MLE algorithm performs better than the original MLE.

Figure~\ref{fig:heatmap} is a heat map of the mean squared error of the smoothed MLE.
The $x$-axis varies the number of samples $n$ from 50 to 5000, and the $y$-axis varies the smoothing parameter $r$ from 0.001 to 1 in log scale.
Lighter color indicates a smaller mean squared error.
The line overlaid on the heat map indicates, for each value of $n$, the value of $r$ with the smallest mean squared error.
As $n$ increases, the optimal value of $r$ decreases, as predicted by our theory.

For small values of $n$---below about 1000---the mean squared error first decreases then increases again as we increase the smoothing parameter $r$.
This confirms the theory in the paper: for small $n$, it is unlikely that we see any samples from the spike, in which case too small values for $r$ cause MLE to overfit.
On the other hand, too large values of $r$ simply add too much noise, and also yields a sub-optimal mean squared error.
The optimal value of $r$ is thus somewhere in between.

The situation changes when $n \gg 1000$, which is 1 over the mass of the spike.
In this case, we expect to typically see samples from the spike, which allows us to estimate the mean highly accurately.
Any smoothing just adds noise, and hence the optimal value of $r$ is close to 0.

\begin{figure}
\begin{subfigure}{.3\textwidth}
    \centering
    \includegraphics[keepaspectratio,width=4.6cm]{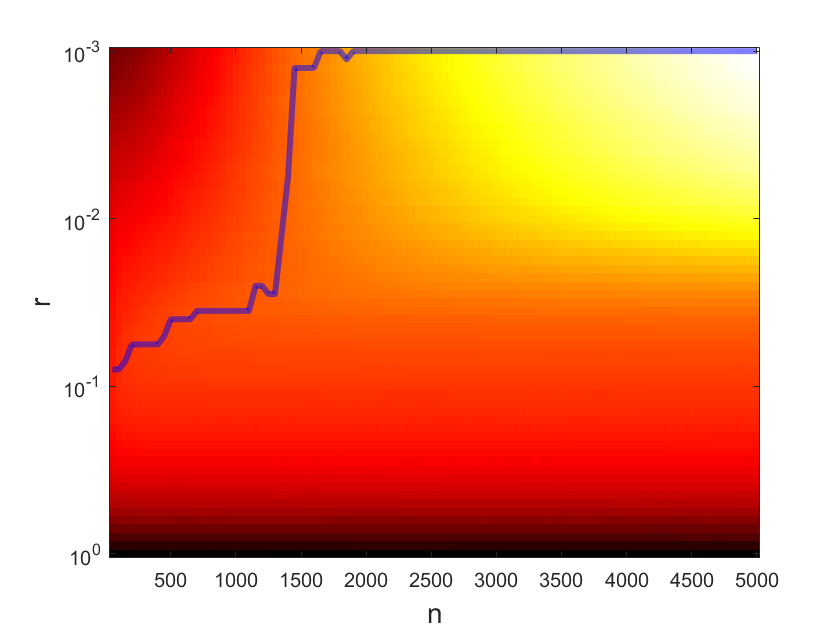}
    \caption{Heat map of mean squared error}
    \label{fig:heatmap}
\end{subfigure}
\begin{subfigure}{.34\textwidth}
    \centering
    \includegraphics[keepaspectratio,width=5.3cm]{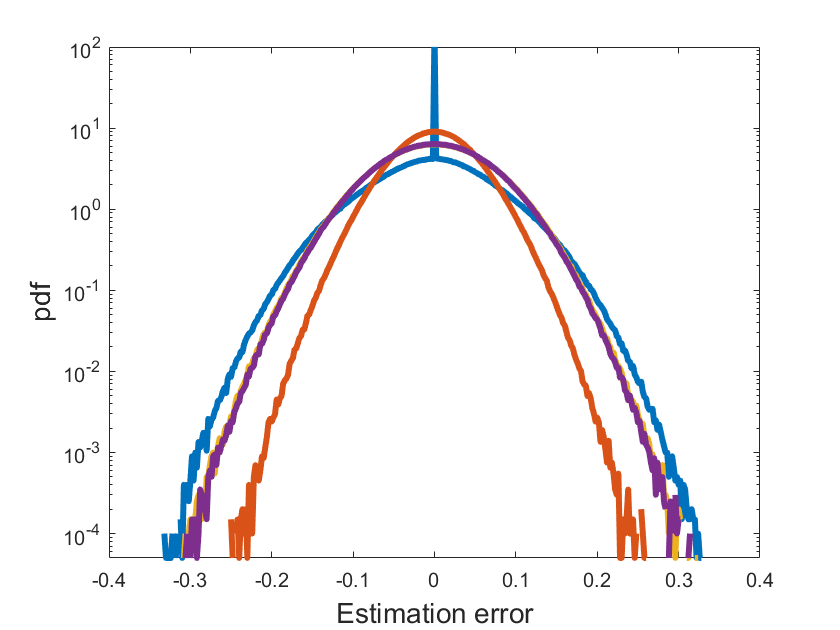}
    \caption{Estimation error distribution,\\ n = 500}
    \label{fig:error_n500}
\end{subfigure}
\begin{subfigure}{.34\textwidth}
    \centering
    \includegraphics[keepaspectratio,width=5.3cm]{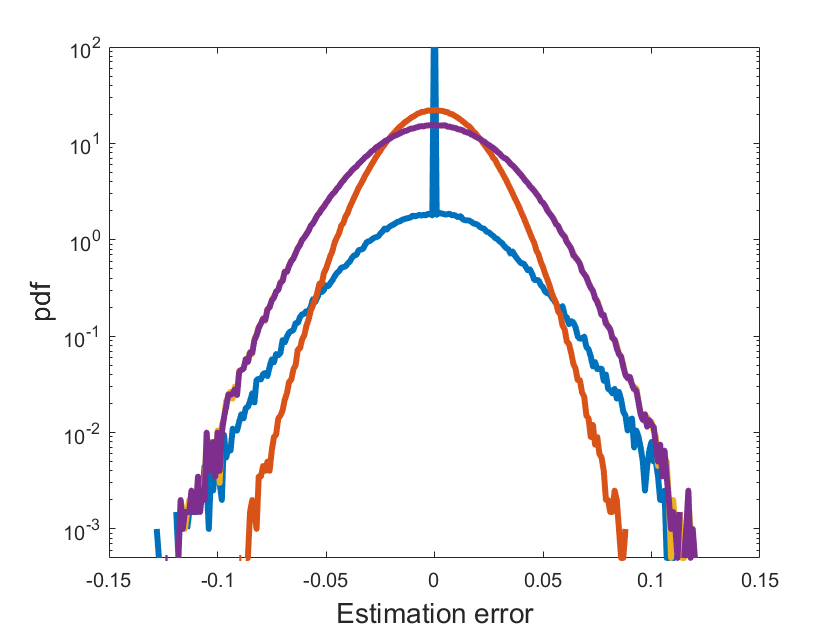}
    \caption{Estimation error distribution,\\ n = 3000}
    \label{fig:error_n3000}
\end{subfigure}
\caption{Experimental results - Spiked Laplace model}
\end{figure}

Figure~\ref{fig:error_n500} picks $n=500$ and compares the distribution of estimation errors across different algorithms: unsmoothed MLE (blue), 0.05-smoothed MLE (orange), empirical mean (yellow), Lee-Valiant (LV) estimator (purple) using $\delta = e^{-5}$.
With only 500 samples, the unsmoothed MLE occasionally sees a sample from the spike, and attains high accuracy, but otherwise has large variance in error, compared with the empirical and LV estimators (which have essentially identical performance, so yellow is overlapped by purple on the plot).
With just 0.05-smoothing, MLE outperforms all other estimators.

Figure~\ref{fig:error_n3000} picks $n=3000$ and compares the same algorithms.
The unsmoothed MLE sees samples from the spike most of the time, and attains high accuracy, vastly outperforming the empirical and LV estimators (again, yellow is overlapped by purple).
The 0.05-smoothed MLE performs worse than the unsmoothed MLE in the typical case, but has better tail behavior.
This plot suggests that the optimal smoothing parameter $r$ in the high probability regime depends on the desired failure probability $\delta$.

%% file: futureWork.tex
\section{Future Directions}


\looseness=-1 
One natural goal is to extend these techniques to estimate the mean of \emph{unknown} distributions by means of a kernel density estimate (KDE), with accuracy dependent on the true distribution's Fisher information.  In general this cannot work, a bias independent of the Fisher information is unavoidable, but for \emph{symmetric} distributions the bias is zero and one can hope for good results.  An asymptotic version of this was shown by Stone~\cite{stone1975adaptive}, and we believe our techniques could get a finite-sample guarantee here.


A second direction is to investigate ways to generalize and simplify our lower bound analysis techniques.
Recall that, while standard ``indistinguishability" bounds based on squared Hellinger distance and KL-divergence are relatively straightforward to apply, they generally lose constant factors.
Our analysis is tight to within a $1+o(1)$-factor, but it requires analyzing several different parameters of the distribution.
One can hope to extend and generalize these techniques to yield a new easy-to-apply bound, similar to those based on Hellinger distance and KL-divergence, that gives $1+o(1)$-factor tightness.


%% file: app_score.tex
\section{Proofs from Section~\ref{sec:rscore}}
\label{app:rscore}
We first prove a utility lemma, Lemma~\ref{lem:p-z}, which we use throughout the rest of the paper.

\begin{restatable}{lemma}{LemPZ}
\label{lem:p-z}
Let $f$ be an arbitrary distribution and let $f_r$ be the $r$-smoothed version of $f$. That is, $f_r = \E_{y \leftarrow f}\left[\frac{1}{\sqrt{2 \pi r^2}} e^{-\frac{(x- y)^2}{2 r^2}}\right]$. Let $s_r$ be the score function of $f_r$. Let $(X, Y, Z_r)$ be the joint distribution such that $Y \sim f$, $Z_r \sim \mathcal N(0, r^2)$ are independent, and $X = Y + Z_r \sim f_r$. We have, for every $\eps > 0$,
  \[
    \frac{f_r(x+\eps)}{f_r(x)} = \E_{Z_r \mid x}\left[e^{\frac{2 \eps Z_r - \eps^2}{2 r^2}}\right]
    \quad
  \text{and in particular}
  \quad
    s_r(x) = \E_{Z_r \mid x}\left[\frac{Z_r}{r^2}\right]
  \]
\end{restatable}

\begin{proof}

  For simplicity of exposition, we only show the case where $f$ has a density. The general case can be proven by, for example, a limit argument. Let $w_r$ be the pdf of $\mathcal N(0, r^2)$. First, we show that for any $x, \eps$ we have
  \begin{align}
    \frac{f_r(x + \eps)}{f_r(x)} = \E_{Z_r \mid x} \left[\frac{w_r(Z_r + \eps)}{w_r(Z_r)}\right]\label{eq:cond}
  \end{align}
  Denote the density of $(x, z, \wh{x})$ by $p(\cdot)$.  Note that
  \[
    p(z \mid x) = \frac{p(x, z)}{p(x)} = \frac{f(x - z) w_r(z)}{f_r(x)}
  \]
  and hence
  \begin{align*}
    f_r(x + \eps)  &= \int_{-\infty}^\infty w_r(z) f(x + \eps - z) \,\d z = \int_{-\infty}^\infty w_r(z + \eps) f(x - z)\, \d z \\
                 &= \int_{-\infty}^\infty p(z \mid x) f_r(x) \frac{w_r(z + \eps)}{w_r(z)} \,\d z\\
                 &= f_r(x) \E_{Z \mid x}\left[\frac{w_r(Z_r + \eps)}{w_r(Z_r)}\right]
  \end{align*}
  proving~\eqref{eq:cond}. 

  Since $w_r(z) = \frac{1}{\sqrt{2 \pi r^2}} e^{-\frac{z^2}{2r^2}}$, this gives
  \[
    \frac{f_r(x + \eps)}{f_r(x)} = \E_{Z \mid x}\left[e^{\frac{2 \eps Z_r - \eps^2}{2r^2}}\right].
  \]
  Taking the derivative with respect to $\eps$ and evaluating at $\eps = 0$,
  \[
    \frac{f_r'(x)}{f_r(x)} = \E_{Z \mid x} \frac{Z_r}{r^2}.
  \]
\end{proof}

We now prove Lemmas~\ref{lem:IrBounded} and~\ref{lem:IrLB}, which upper and lower bound the $r$-smoothed Fisher information $\I_r$ respectively.

\LemIrBounded*

\begin{proof}
  Using Lemma~\ref{lem:p-z} and Jensen's inequality,
\[
      I_r = \E_x[s_r^2(x)] = \E_x[(\E_{Z_r|x} Z_r/r^2)^2] \le \E_x[\E_{Z_r|x} Z_r^2/r^4] = 1/r^2 \qedhere\]
\end{proof}

\LemIrLB*

\begin{proof}
  First, observe that $f_r$ is a smooth distribution in the sense that it is differentiable, and furthermore, its derivative is continuous.
  Thus, letting $R$ be the $30^{\text{th}}$-$70^{\text{th}}$ percentile range of $f_r$.
  Then, by a known result~\cite{Shevlyakov:2011} (Section 3.1), $I_r \gtrsim 1/R^2$.
  
  Furthermore, it is easy to see via a coupling argument that $R = \IQR + O(r)$.
  The lemma statement follows.
\end{proof}

Next, we prove another utility lemma, which states that the derivative of the score function cannot be too small for an $r$-smoothed distribution.
Phrased differently, the score function of an $r$-smoothed distribution cannot decrease fast.

\begin{lemma}\label{lem:u-lowerderivative}
  $s_r'(x) \geq -1/r^2$ for all $x$, where $s_r$ is the score function of $f_r$, the $r$-smoothed version of distribution $f$.
\end{lemma}
\begin{proof}
  By taking the derivative of Lemma~\ref{lem:p-z} in $\eps$,
  \[
    \frac{f_r'(x+\eps)}{f_r(x)} = \E_{Z \mid x}\left[e^{\frac{2 \eps Z - \eps^2}{2 r^2}}\frac{Z_r - \eps}{r^2}\right]
  \]
  Hence
  \[
    s_r(x+\eps) = \frac{f_r'(x+\eps)}{f_r(x+\eps)} =     \frac{f_r'(x+\eps)}{f_r(x)}\frac{f_r(x)}{f_r(x+\eps)} = \frac{\E_{Z_r \mid x}\left[e^{\frac{2 \eps Z_r - \eps^2}{2 r^2}}\frac{Z_r - \eps}{r^2}\right]}{\E_{Z_r \mid x}\left[e^{\frac{2 \eps Z_r - \eps^2}{2 r^2}}\right]}.
  \]
  For $\eps > 0$, since $e^{\frac{2 \eps Z_r - \eps^2}{2 r^2}}$ and
  $\frac{Z_r - \eps}{r^2}$ are monotonically increasing in $Z_r$, and the
  former is nonnegative, they are positively correlated:
  \[
    \E_{Z_r \mid x}\left[e^{\frac{2 \eps Z_r - \eps^2}{2 r^2}}\frac{Z_r -
      \eps}{r^2}\right] \geq \E_{Z_r \mid x}\left[e^{\frac{2 \eps Z_r - \eps^2}{2
        r^2}}\right]\E_{Z_r \mid x}\left[\frac{Z_r - \eps}{r^2}\right]
  \]
  Hence
  \[
    s_r(x + \eps) \geq \E_{Z_r \mid x}\left[\frac{Z_r - \eps}{r^2}\right] = s_r(x) - \frac{\eps}{r^2}.
  \]
  or (taking $\eps \to 0$), $s_r'(x) \geq -\frac{1}{r^2}$.
\end{proof}

Lastly, we prove the concentration of empirical score function.
The way we do so is to show (Lemma~\ref{lem:offcenter-score-moments}) that the $k^{\text{th}}$ absolute moment of the score function is upper bounded according to the standard moment bounds for sub-Gamma distributions.
As a corollary (Corollary~\ref{cor:scoreestimate}), we get that the scores have sub-Gamma concentration.

As a utility lemma, we bound the moments of the score function when the score function is aligned with the distribution, instead of being misaligned by some $\eps$ distance.

\begin{lemma}\label{lem:score-moments}
  Let $s_r$ be the score function of an $r$-smoothed distribution $f_r$ with Fisher information $\I_r$.
  Then, for $k \ge 3$,
  \[
    \E_x[\abs{s_r(x)}^k] \leq (1.6/r)^{k-2} k^{k/2}\I_r
  \]
\end{lemma}
\begin{proof}

  For any $x, \eps$, by Lemma~\ref{lem:p-z} and Jensen's inequality,
  \[
    f_r(x + \eps) \geq f_r(x) e^{\eps s_r(x) - \frac{\eps^2}{2r^2}}.
  \]
  Setting $\eps = \pm r$ with sign matching $s_r(x)$, we have that
  \[
    f_r(x + r \text{sign}(s_r(x))) \geq  f_r(x) e^{r \abs{s_r(x)}} / \sqrt{e}.
  \]
  We also have, from Lemma~\ref{lem:u-lowerderivative}, that
  \[
    s_r(x-r) \leq s_r(x) + 1/r
  \]
   and
  \[
    s_r(x+r) \geq s_r(x) - 1/r.
  \]
  In other words,
  \[
    \abs{s_r(x + r \sign(s_r(x)) )} \geq \abs{s_r(x)} - 1/r.
  \]
  Therefore, for any $k \geq 2$, and $\abs{s_r(x)} > \alpha/r$ for $\alpha := 2 + 1.2 \sqrt{k}$,
  \begin{align*}
    f_r(x + r\text{sign}(s_r(x)))\abs{s_r(x+r\text{sign}(s_r(x)))}^k
    &\geq \frac{1}{\sqrt{e}}f_r(x) e^{r \abs{s_r(x)}} (\abs{s_r(x)}-1/r)^k\\
    &= f_r(x) \abs{s_r(x)}^k \cdot \left(\frac{1}{\sqrt{e}} e^{r \abs{s_r(x)}} (1 - \frac{1}{r \abs{s_r(x)}})^k\right)\\
    &\geq f_r(x) \abs{s_r(x)}^k \cdot \left(\frac{1}{\sqrt{e}} e^{\alpha - 1.4 \frac{k}{\alpha}}\right)\\
    &\geq f_r(x) \abs{s_r(x)}^k \cdot 4.
  \end{align*}
  Therefore
  \begin{align}
   f_r(x) \abs{s_r(x)}^k &\leq \frac{1}{4} \left(f_r(x - r)\abs{s_r(x-r)}^k + f_r(x + r)\abs{s_r(x+r)}^k\right)
  \end{align}
  whenever $k \geq 2$ and $\abs{s_r(x)} \geq \alpha/r$.  Integrating
  this,
  \begin{align*}
    \E[s_r^k(x)] = \int_{-\infty}^\infty f_r(x) \abs{s_r(x)}^k \,\d x &= 2 \int_{-\infty}^\infty f_r(x) \abs{s_r(x)}^k - \frac{1}{4} f_r(x - r)\abs{s_r(x-r)}^k - \frac{1}{4} f_r(x + r)\abs{s_r(x+r)}^k \,\d x\\
                                               &\leq 2 \int_{-\infty}^\infty f_r(x) \abs{s_r(x)}^k 1_{\abs{s_r(x)} < \alpha/r} \,\d x\\
                                               &\leq 2 \int_{-\infty}^\infty f_r(x) \abs{s_r(x)}^2 (\alpha/r)^{k-2} 1_{\abs{s_r(x)} < \alpha/r} \,\d x\\
                                               &\leq 2  (\alpha/r)^{k-2} \E_x[s_r^2(x)] = 2 (\alpha/r)^{k-2} \I_r
  \end{align*}
  Finally, we observe for any $k \geq 2$ that
  \[
    2 (1.2\sqrt{k} + 2)^{k-2} \leq k^{k/2} \cdot 1.6^{k-2}
  \]
  giving the lemma.

\end{proof}

The proof of Lemma~\ref{lem:offcenter-score-moments} has the same logical structure as the proof of Lemma~\ref{lem:score-moments}, but has further subtleties.
The following two lemmas generalize the first step in the proof of Lemma~\ref{lem:score-moments}.

\begin{lemma}\label{lem:eps_pos_s_pos_helper}
Let $s_r$ be the score function of an $r$-smoothed distribution $f_r$ with Fisher information $\I_r$.
    For any $x$, $k \ge 3$ and $0 \le \eps \le r/2$, if $s_r(x+\eps) \ge \max(2 \sqrt{k} + 2, 9.5)/r$, then
    \begin{align*}
        f_r(x) |s_r(x+\eps)|^k \le \frac{1}{5} \max\left(f_r(x - \eps) \abs{s_r(x - \eps) }^k, f_r(x + \eps+r)\abs{s_r(x +\eps+ r)}^k \right)
    \end{align*}
    
\end{lemma}
\begin{proof}
Let $\alpha := \frac{f_r(x)}{f_r(x+\eps)}$. By Lemma~\ref{lem:p-z}, we have
\begin{equation}\label{eq:alphadef}
\alpha = \E_{Z_r | x+\eps}\left[e^{\frac{-2 \eps Z_r - \eps^2}{2 r^2}}\right]
\end{equation}
We will consider two cases.
\paragraph{When $\log \alpha \le \frac{3}{4}r s_r(x+\eps) - 2$.}
First, by Lemma \ref{lem:p-z} and Jensen's inequality, we have
$$\frac{f_r(x + \eps+ r)}{f_r(x+\eps)} \ge e^{r s_r(x+\eps) - 1/2}$$
We also have, by Lemma~\ref{lem:u-lowerderivative},
$$s_r(x + \eps+ r) \ge s_r(x+\eps) - 1/r$$

So,
\begin{align*}
f_r(x + \eps+ r) |s_r(x +\eps+ r)|^k &\ge f_r(x+\eps) |s_r(x+\eps)|^k  e^{r s_r(x) - 1/2} \left(1 - \frac{1}{r s_r(x+\eps)} \right)^k\\
&\ge f_r(x+\eps) |s_r(x+\eps)|^k e^{r s_r(x+\eps) - \frac{k}{r s_r(x+\eps) - 1} - 1/2}
\end{align*}

Since $s_r(x+\eps) \ge (2 \sqrt{k} + 2)/r$,
$$f_r(x + \eps+ r) |s_r(x +\eps+ r)|^k \ge f_r(x+\eps) |s_r(x+\eps)|^k e^{\frac{3}{4}r s_r(x+\eps)}$$

So, since $$\alpha = \frac{f_r(x)}{f_r(x+\eps)} \le e^{\frac{3}{4} rs_r(x+\eps) - 2} $$ we have $$f(x) |s_r(x+\eps)|^k = \alpha f_r(x+\eps) |s_r(x+\eps)|^k \le \frac{1}{5} f(x+\eps+r) |s_r(x +\eps+ r)|^k$$
\paragraph{When $\log \alpha > \frac{3}{4} r s_r(x+\eps) - 2$}
Evaluating \eqref{eq:alphadef} at $x-\eps$ gives
\[ \frac{f_r(x-\eps)}{f_r(x)} = \E_{Z_r \, | \, x}\left[e^{\frac{-2\eps Z_r-\eps^2}{2r^2}}\right]\]
Taking derivative with respect to $\eps$, we have $$\frac{f_r'(x - \eps)}{f_r(x)} = \E_{Z_r | x}\left[\frac{(Z_r + x)}{r^2} e^{\frac{-2\eps Z_r - \eps^2}{2r^2}}\right]$$
and so by evaluating at $x+\eps$ (to ``shift back")
$$\frac{f_r'(x)}{f_r(x+\eps)} = \E_{Z_r | x+\eps}\left[\frac{(Z_r + x+\eps)}{r^2} e^{\frac{-2\eps Z_r - \eps^2}{2r^2}}\right]$$
Define $y = e^{\frac{-2 \eps Z_r - 2\eps^2}{2 r^2}}$, so that $\E_{Z_r\mid x+\eps}[y] = \alpha e^{\frac{\eps^2}{2 r^2}}$, and
$$\frac{(Z_r + \eps)}{r^2} e^{\frac{-2\eps Z_r - \eps^2}{2r^2}} = -\frac{e^{\frac{\eps^2}{2 r^2}}}{\eps} y \log y$$
is concave, so by Jensen's inequality
$$\frac{f'(x)}{f(x+\eps)} \le -\frac{e^{\eps^2/(2 r^2)}}{\eps} (e^{-\frac{\eps^2}{2 r^2}} \alpha) \log (e^{-\frac{\eps^2}{2 r^2}} \alpha) = - \frac{\alpha \log \alpha}{\eps} + \frac{\alpha \eps}{2 r^2}$$
So,
$$s_r(x) = \frac{f_r'(x)}{f_r(x)} \le - \frac{\log \alpha}{\eps} + \frac{\eps}{2 r^2}$$
Finally, we move to consider the point $x - \eps$.
By Lemma~\ref{lem:u-lowerderivative}, we have
$$s_r(x - \eps) \le s_r(x) + \eps/r^2 \le -\frac{\log \alpha}{\eps} + \frac{3 \eps}{2r^2}$$

By Lemma~\ref{lem:p-z},
$$\frac{f(x - \eps)}{f(x+\eps)} = \E_{Z_r \mid x+\eps}\left[e^{-\frac{4 \eps Z_r - 4 \eps^2}{2 r^2}} \right] = \E_{Z_r\mid x+\eps}[y^2] \ge \E_{Z_r \mid x+\eps}[y]^2 = \alpha^2 e^{-\frac{\eps^2}{r^2}}$$
Since $\log \alpha \ge \frac{3}{4} s_r(x+\eps) - 2$,
$$-s_r(x - \eps) \ge \frac{\frac{3}{4}r s_r(x+\eps) - 2}{\eps} - \frac{3 \eps}{2 r^2} \ge \frac{3}{2} s_r(x+\eps) - \frac{4.75}{r} \ge s_r(x)$$
where the second inequality comes from the fact that $\frac{3}{4}rs_r(x+\eps)-2 > 0$ and so the function is decreasing in $\eps$, with minimum evaluated at $\eps = r/2$.

Thus, we have
$$f_r(x - \eps) |s_r(x - \eps)|^k \ge \alpha e^{-\eps^2/r^2} f_r(x) |s_r(x+\eps)|^k $$
Since our assumptions give $\alpha e^{-\eps^2/r^2} \ge e^{5.125} e^{-1/4} \ge 5$, we get the result.
\end{proof}

\begin{lemma}\label{lem:eps_neg_s_pos_helper}
Let $s_r$ be the score function of an $r$-smoothed distribution $f_r$ with Fisher information $\I_r$.

For any $x$, $k \ge 3$ and $-r/2 \le \eps \le 0$, if $s_r(x+\eps) \ge \alpha/r$ for $\alpha = 2 + 1.2 \sqrt{k}$, then we have
\[f_r(x) |s_r(x+\eps)|^k \le \frac{1}{4}\left(f_r(x -  r ) |s_r(x+\eps - r)|^k + f_r(x + r) |s_r(x + \eps + r)|^k\right)\]

As an immediate corollary, the statement is true also when $\eps \in [0, r/2)$ and $s_r(x) \le -\alpha/r$.
\end{lemma}

\begin{proof}
  For any $x, \kappa$, by Lemma~\ref{lem:p-z} and Jensen's inequality,
  \[
    f_r(x + \kappa) \geq f_r(x) e^{\kappa s_r(x) - \frac{\kappa^2}{2r^2}}.
  \]
  So, setting $\kappa = r$, we have
  \[
    f_r(x + r ) \ge f_r(x) e^{r  s_r(x)}/\sqrt{e} 
  \]
  By Lemma~\ref{lem:u-lowerderivative}, we have that 
  $$s_r(x + \eps + r) \ge s_r(x + \eps) - 1/r$$
  Since our right hand side is positive by assumption, this is equivalently stated as
  $$|s_r(x + \eps + r ))| \ge |s_r(x + \eps)| - 1/r$$
    
When $\eps < 0$, we have, by Lemma~\ref{lem:u-lowerderivative}, and since $|\eps| \le r$, $s_r(x) \ge s_r(x + \eps) - 1/r $. So,
\begin{align*}
    f_r(x + r ) |s_r(x + \eps + r )|^k 
    &\geq \frac{1}{\sqrt{e}} f_r(x) e^{r s_r(x)} (\abs{s_r(x+\eps)}-1/r)^k\\
    &\ge \frac{1}{\sqrt{e}} f_r(x) e^{r  (s_r(x + \eps) - 1/r)} (\abs{s_r(x+\eps)}-1/r)^k\\
    &\ge f_r(x) |s_r(x + \eps)|^k \left(\frac{1}{\sqrt{e}} e^{r|s_r(x + \eps)| - 1} \left(1 - \frac{1}{r |s_r(x + \eps)|} \right)^k \right)\\
    &\ge f_r(x) |s_r(x + \eps)|^k \cdot \left(e^{-3/2} e^{\alpha - 1.4k/\alpha} \right)\\
    &\ge f_r(x) |s_r(x + \eps)|^k \cdot 4 \quad \text{since $k \ge 3$}\\
  \end{align*}
\end{proof}

We are now ready to prove Lemma~\ref{lem:offcenter-score-moments}, which states that the distribution of the score function $s_r(x+\eps)$ where $x \sim f_r$ is a sub-Gamma distribution.
Corollary~\ref{cor:scoreestimate} then states that the average of many score function samples is well-concentrated, following sub-Gamma concentration.

\begin{lemma}\label{lem:offcenter-score-moments}
Let $s_r$ be the score function of an $r$-smoothed distribution $f_r$ with Fisher information $\I_r$.
  Then, for $k \ge 3$ and $|\eps| \le r/2$,
  \[
    \E_x[\abs{s_r(x+\eps)}^k] \leq \frac{k!}{2} (15/r)^{k-2} \max(\E_x[s^2_r(x + \eps)], \I_r)
  \]
  Equivalently, $s_r(x+\eps)$ is a sub-Gamma random variable.
  \[
    s_r(x+\eps) \in \Gamma(\max(\E_x[s_r^2(x+\eps)],\I_r), 15/r).
  \]
\end{lemma}

\begin{proof}
Without loss of generality we only show the $\eps \ge 0$ case.

Using Lemma~\ref{lem:eps_pos_s_pos_helper} and Lemma~\ref{lem:score-moments}, we have
\begin{align*}
    &\int_{-\infty}^\infty f_r(x - \eps) |s_r(x)|^k \1_{s_r(x) > \max(2 \sqrt{k} + 2, 9.5)/r} \,\d x\\
    \le\;& \frac{1}{5}\int_{-\infty}^\infty f_r(x - 2 \eps)|s_r(x - 2\eps)|^k + f_r(x+r)|s_r(x+r)|^k\,\d x \\
    =\;& \frac{2}{5} \E[|s_r(x)|^k]\\
    \le\;& \frac{2}{5}(1.6/r)^{k-2}k^{k/2}\I_r
\end{align*}
We can start bounding the $k^{\text{th}}$ moment quantity in the lemma:
\begin{align*}
     &\;\;\E[|s_r(x+\eps)|^k]\\
     =\;& \int_{-\infty}^\infty f_r(x-\eps) \abs{s_r(x)}^k \,\d x\\
     =\;& 2 \int_{-\infty}^\infty f_r(x-\eps) \abs{s_r(x)}^k - \frac{1}{4} f_r(x -\eps- r)\abs{s_r(x-r)}^k - \frac{1}{4} f_r(x -\eps + r)\abs{s_r(x+r)}^k \,\d x\\
     \le\;& 2 \int_{-\infty}^\infty f_r(x-\eps) |s_r(x)|^k \1_{s_r(x) \ge -\max(2\sqrt{k} + 2, 9.5)/r}\,\d x
\end{align*}
where the last inequality follows from (a slight weakening of) Lemma~\ref{lem:eps_neg_s_pos_helper}. Now, by the previous claim, we get that
\begin{align*}
&\;\;\E[|s_r(x+\eps)|^k]\\
\le\;& 2 \int_{-\infty}^\infty f_r(x - \eps) |s_r(x)|^k \1_{|s_r(x)| \le \max(2 \sqrt{k} + 2, 9.5)/r} \,\d x + \frac{4}{5} (1.6/r)^{k-2}k^{k/2}\I_r\\
\le\;& 2 \int_{-\infty}^\infty f_r(x-\eps) |s_r(x)|^2 (\max(2 \sqrt{k}+2, 9.5)/r)^{k-2} \1_{|s_r| \le \max(2 \sqrt{k} + 2, 9.5)/r} \,\d x+ \frac{4}{5} (1.6/r)^{k-2}k^{k/2}\I_r\\
\le\;& 2 (\max(2 \sqrt{k} + 2, 9.5)/r)^{k-2} \E_x[|s_r(x + \eps)|^2] + \frac{4}{5} (1.6/r)^{k-2}k^{k/2}\I_r\\
\le\;& 2 k^{k/2} (2.5/r)^{k-2} \E_x[|s_r(x + \eps)|^2] + \frac{4}{5} (1.6/r)^{k-2}k^{k/2}\I_r\\
\le\;& 3 k^{k/2} (2.5/r)^{k-2} \max(\E_x[|s_r(x + \eps)|^2], \I_r)\\
\le\;& \frac{k!}{2} (15/r)^{k-2} \max(\E_x[|s_r(x + \eps)|^2], \I_r)
\end{align*}

\end{proof}

\CorScoreEstimate*
\begin{proof}
Since $\hat s(\lam + \eps) = \frac{1}{n} \sum_{i=1}^n s_r(y_i - \lam - \eps) = \frac{1}{n} \sum_{i=1}^n s_r(x_i - \eps)$, we know that by Lemma~\ref{lem:offcenter-score-moments} and the standard algebra of sub-Gamma distributions that $\hat{s}(\lambda+\eps) \in \Gamma(\frac{1}{n}\max(\E_x[s_r^2(x+\eps)],\I_r), 15/r)$.
The corollary then follows from the standard Bernstein inequality for sub-Gamma distributions~\cite{boucheron}.
\end{proof}

%% file: app_localMLE.tex
\section{Proofs omitted in Section~\ref{sec:localMLE}}
\label{app:localMLE}

We first give the proof of Theorem~\ref{thm:simpleLocalMLE}, assuming Theorem~\ref{thm:localMLE}.

\begin{proof}[Proof of Theorem~\ref{thm:simpleLocalMLE}]
  It suffices to show that conditions 2) and 3) in the corollary statement implies that each of the following terms from Theorem~\ref{thm:localMLE} is $1+O(1/\gamma)$:
  \begin{itemize}
      \item $1+1/\log\frac{1}{\delta}$: Note that $\I_r \le \frac{1}{r^2}$ by Lemma~\ref{lem:IrBounded} and so condition 3) implies that $1/\log\frac{1}{\delta} \le (\log\log\frac{1}{\delta})/\log \frac{1}{\delta} \le \frac{1}{\gamma}$
      \item $1+\rho_r$: $\sqrt{1+O(1/\gamma)} = 1+O(1/\gamma)$. It suffices to check that
      $ \left(\frac{2\log\frac{4\log\frac{1}{\delta}}{r^2 \I_r(1-\frac{\beta}{\gamma})\delta}}{n}\right)^{\frac{1}{4}} = O(1/\sqrt{\gamma})$.
      The fact that $\log \frac{\log \frac{1}{\delta}}{\delta} = O(\log\frac{1}{\delta})$ together with condition 3) imply that the quantity is bounded by $O\left(\frac{(1+O(\gamma))\log\frac{1}{\delta}}{n}\right)^{\frac{1}{4}}$, which in turn is bounded by $O(1/\sqrt{\gamma})$ by condition 2).
      \item $1/(1-\beta/\gamma) \le 1+O(1/\gamma)$ since $\beta$ is a constant
      \item $\sqrt{1+\frac{\log\frac{4\log\frac{1}{\delta}}{r^2 \I_r (1-\frac{\beta}{\gamma})}}{\log\frac{1}{\delta}}}$: Note that $\frac{\log \frac{4\log\frac{1}{\delta}}{1-\frac{\beta}{\gamma}}}{\log\frac{1}{\delta}} = O(\frac{\log\log\frac{1}{\delta}}{\log\frac{1}{\delta}})= O(1/\gamma)$ as before.
      Also, condition 3) implies that $(\log\frac{1}{r^2 \I_r})/(\log \frac{1}{\delta}) \le (\log\frac{1}{r^2 \I_r})(\log \log \frac{1}{\delta})/(\log \frac{1}{\delta}) \le 1/\gamma$.
  \end{itemize}
\end{proof}

The rest of this appendix is on proving Lemma~\ref{lem:localMLE}, which via reparameterization gives Theorem~\ref{thm:localMLE}.

We first show a utility lemma (Lemma~\ref{lem:smoothed}), before using it to prove Lemmas~\ref{lem:grad_expectation_theta} and~\ref{lem:var_close_to_fisher}, which bound the expectation and variance of the empirical score function.
After that, we prove Lemma~\ref{lem:localMLE}.

\begin{lemma}\label{lem:smoothed}
  Let $w_r$ be a Gaussian with standard deviation $r$, $f$ be an
  arbitrary probability distribution, and $f_r$ be the $r$-smoothed version of $f$.
  Define
  \[
    \Delta_\eps(x) := \frac{f_r(x+\eps) - f_r(x) - \eps f_r'(x)}{f_r(x)}.
  \]
  Then for any $\abs{\eps} \leq r/2$,
  \[
    \E_{x \sim f_r}\left[ \Delta_{\eps}(x)^2\right] \lesssim \frac{\eps^4}{r^4}.
  \]
\end{lemma}
\begin{proof}
    By Lemma~\ref{lem:p-z}, we have
  \[
    \Delta_{\eps}(x) = \frac{f_r(x +\eps) - f_r(x) - \eps f_r'(x)}{f_r(x)} = \E_{Z_r \mid x} (e^{\frac{2 \eps Z_r - \eps^2}{2r^2}} - 1 - \frac{\eps Z_r}{r^2} ).
  \]
  Define
  \[
    \alpha_\eps(z) := e^{\frac{2 \eps z - \eps^2}{2r^2}} - 1 - \frac{\eps z}{r^2}.
  \]
  We want to bound
  \begin{align}
     &\E_X\left[\Delta_{\eps}(x)^2\right]\notag\\
     &= \E_X\left[\E_{Z_r \mid X} \left[\alpha_\eps(Z_r) \right]^2\right]\notag\\
     &\leq \E_{X,Z_r}\left[\left(\alpha_\eps(Z_r) \right)^2\right]\notag\\
     &= \E_{Z_r \sim N(0, r^2)} \left(\alpha_\eps(Z_r) \right)^2.\label{eq:Zgoal}
  \end{align}
  Finally, we bound this term~\eqref{eq:Zgoal}.

  When $\abs{\eps z} \leq  r^2$, we have by a Taylor expansion that
  \[
    e^{\frac{2 \eps z - \eps^2}{2r^2}} = 1 + \frac{\eps z}{r^2} - \frac{\eps^2}{2r^2} + O\left(\left(\frac{2 \eps z - \eps^2}{2r^2}\right)^2\right)
  \]
  and so
  \begin{align*}
    \left|\alpha_\eps(z)\right| \lesssim \frac{\eps^2}{r^2} + \left(\frac{\eps z}{r^2}\right)^2
  \end{align*}
  or
  \begin{align}
    \E_{Z_r \sim N(0, r^2)} \left(\alpha_\eps(Z_r)^2 \cdot \1_{\abs{\eps Z_r} \leq r^2}\right)^2 \lesssim \frac{\eps^4}{r^4}.\label{eq:zsmall}
  \end{align}
  On the other hand, for $\abs{\eps z} \geq r^2$,
  \[
    \abs{\alpha_\eps(z) } \leq e^{\abs{\frac{\eps z}{r^2}}}
  \]
  so
  \begin{align*}
    \E_{Z_r \sim N(0, r^2)} \left(\alpha_\eps(Z_r)^2 \cdot \1_{\abs{\eps Z_r} \geq r^2}\right)^2
    &\leq 2\int_{\abs{r^2/\eps}}^\infty \frac{1}{\sqrt{2 \pi r^2}}e^{\frac{2\abs{\eps z}}{r^2}} e^{-\frac{z^2}{2r^2}} \,\d z\\
    &= 2e^{2 \eps^2/r^2} \int_{\abs{r^2/\eps}}^\infty \frac{1}{\sqrt{2 \pi r^2}}e^{-\frac{(z - 2\abs{\eps})^2}{2r^2}} \,\d z\\
    &\leq 2 \sqrt{e} \Pr[z \geq r^2/\abs{\eps} - 2\abs{\eps}]\\
    &\lesssim e^{- \frac{(\abs{r^2/\eps} - 2 \abs{\eps})^2}{2r^2}}\\
    &\leq e^{- \frac{r^2}{8\eps^2}} \lesssim \frac{\eps^4}{r^4}.
  \end{align*}
  Which combines with~\eqref{eq:Zgoal} and~\eqref{eq:zsmall} to give the result.
\end{proof}

We are now ready to prove Lemma~\ref{lem:grad_expectation_theta}, which bounds the expectation of the empirical score function.

\begin{restatable}{lemma}{LemGradExpectation}
  \label{lem:grad_expectation_theta}
  Suppose $f_r$ is an $r$-smoothed distribution with Fisher information $\I_r$.
  Then, for any $\abs{\eps} \leq r/2$, the expected score
  $\E_{x\sim f_r}\left[s_r(x+\eps) \right]$ satisfies
  \[
    \E_{x\sim f_r}\left[s_r(x+\eps)\right] = - \I \eps + \Theta\left(\sqrt{\I_r} \frac{\eps^2}{r^2}\right)
  \]
  
\end{restatable}

\begin{proof}
By definition of $s_r$,
  \begin{align*}
    \E_{x\sim f_r}\left[s_r(x+\eps) \right] &= \int_{-\infty}^{\infty}  \frac{f_r(x) f_r'(x+\eps)}{f_r(x+\eps)} \,\d x\\
    &= \int_{-\infty}^{\infty} f_r'(x) \frac{f_r(x-\eps) - f_r(x)}{f_r(x)} \,\d x
  \end{align*}
  Since by definition of $\I_r$,
  \[
    \I_r:= \int_{-\infty}^{\infty}  \frac{f_r'(x)^2}{f_r(x)} \,\d x,
  \]
  \begin{align*}
   \E_{x\sim f_r}\left[s_r(x+\eps) \right]  + \eps \I_r&= \int_{-\infty}^{\infty} \frac{f_r'(x)}{f_r(x)} (f_r(x-\eps) - f_r(x) + \eps f_r'(x)) \,\d x\\
    & = \E\left[ s_r(x) \cdot \Delta_{-\eps}(x) \right]
  \end{align*}
  where $\Delta_\eps(x) := \frac{f_r(x+\eps) - f_r(x) - \eps f_r'(x)}{f_r(x)}$.  Thus
  \begin{align*}
    \left(\E_{x\sim f_r}\left[s_r(x+\eps) \right]  + \eps \I_r\right)^2
    & \leq \E\left[ s_r(x)^2\right] \E\left[ \Delta_{-\eps}(x)^2 \right]\\
    &= \I_r\E\left[ \Delta_{-\eps}(x)^2 \right]
  \end{align*}
  By Lemma~\ref{lem:smoothed}, we have that
  \[
    \E\left[ \Delta_{\eps}(x)^2 \right] \lesssim \frac{\eps^4}{r^4}.
  \]
  and so
  \[
    \left|\E_{x\sim f_r}\left[s_r(x+\eps) \right]  + \eps \I_r\right| \lesssim \sqrt{ \I_r} \frac{\eps^2}{r^2}
  \]
  as desired.
\end{proof}

\begin{restatable}{lemma}{LemVarFisher}
  \label{lem:var_close_to_fisher}
  Suppose $f_r$ is an $r$-smoothed distribution with Fisher information $\I_r$. Then, for any $|\eps| \le r/2$, the second moment of the score satisfies
\[
  \E_{x\sim f_r}\left[ s_r^2(x + \eps)\right] \le \I_r + O\left( \frac{\eps}{r} \I_r \sqrt{\log \frac{1}{r^2\I_r}}\right)
\]
\end{restatable}

\begin{proof}

  We have that
  \begin{align*}
   \E_{x\sim f_r}\left[ s_r^2(x + \eps)\right] &= \int_{-\infty}^{\infty} f_r(x) \left(\frac{f_r'(x+\eps)}{f_r(x+\eps)}\right)^2 \,\d x\\
                          &= \int_{-\infty}^{\infty} f_r(x-\eps) \left(\frac{f_r'(x)}{f_r(x)}\right)^2 \,\d x\\
                          &= \I_r + \int_{-\infty}^{\infty} (f_r(x-\eps) - f_r(x)) \left(\frac{f_r'(x)}{f_r(x)}\right)^2 \,\d x
  \end{align*}
  By Lemma~\ref{lem:p-z}, we have
  \[
    \frac{f_r(x - \eps) - f_r(x)}{f_r(x)} = \E_{Z_r \mid x} \left(e^{\eps Z_r/r^2 - \eps^2/2r^2} -
    1\right).
  \]

  We have that
  \[
    \frac{f_r'(x)}{f_r(x)} = \E_{Z_r \mid x} \frac{Z_r}{r^2}.
  \]
  so that we need to bound
  \begin{align}
    \E_{x\sim f_r}\left[ s_r^2(x + \eps)\right]-\I_r = \E_x\left[ \left(\E_{Z_r \mid x} (e^{\eps Z_r/r^2 - \eps^2/2r^2} - 1)\right)\left(\E_{Z_r \mid x} \frac{Z_r}{r^2}\right)^2 \right].\label{eq:varbound-goal}
  \end{align}
  We can get bound this as follows.
  The standard rearrangement inequality states that, if $g, h$ are
  monotonically non-decreasing functions,
  $\E_z[g(z)]\E_z[h(z)] \leq \E[g(z) h(z)]$.
  Therefore, for any $x$ and
  parameter $\alpha \lesssim r/\eps$,

  \begin{align*}
    \left(\E_{Z_r \mid x} (e^{\eps Z_r/r^2 - \eps^2/2r^2} - 1)\right)\left(\E_{Z_r \mid x} \frac{Z_r}{r^2}\right)^2
    &\leq \left(O(\eps \alpha/r) + \E_{Z_r \mid x} \1_{Z_r > \alpha r}(e^{\eps Z_r/r^2} - 1)\right)\left(\E_{Z_r \mid x} \frac{Z_r}{r^2}\right)^2\\
    &\leq O(\eps \alpha/r)\left(\E_{Z_r \mid x} \frac{Z_r}{r^2}\right)^2 + \left(\E_{Z_r \mid x} \1_{Z_r > \alpha r} \frac{Z_r}{r^2} (e^{\eps Z_r/r^2} - 1)\right)\left(\E_{Z_r \mid x} \frac{Z_r}{r^2}\right)\\
    &\leq O(\eps \alpha/r)\left(\E_{Z_r \mid x} \frac{Z_r}{r^2}\right)^2 + \E_{Z_r \mid x} \left( \1_{Z_r > \alpha r} \frac{Z_r^2}{r^4} (e^{\eps Z_r/r^2} - 1)\right)
  \end{align*}
  Thus
  \begin{align*}
    \E_{x\sim f_r}\left[ s_r^2(x + \eps)\right] - \I_r \leq O(\eps \alpha/r) \I_r + \E_{Z_r}\left[\1_{Z_r > \alpha r} \frac{Z_r^2}{r^4} (e^{\eps Z_r/r^2} - 1)\right].
  \end{align*}
  The contribution to the second term from $|Z_r| \lesssim r^2/\eps$ is
  \[
    \eps \E_{Z_r}[\1_{Z_r > \alpha r} \frac{Z_r^3}{r^6}] \lesssim \frac{\eps \alpha^3}{r^3} e^{-\alpha^2/2}.
  \]
  The contribution from $|Z_r| \gtrsim r^2/\eps$ is negligible (as in the proof of Lemma~\ref{lem:smoothed}) because the chance of such $Z_r$ is
  exponentially small:
  \begin{align*}
      \E_{Z_r}\left[\1_{Z_r \gtrsim \max(r^2/\eps,\alpha r)} \frac{Z_r^2}{r^4} (e^{\eps Z_r/r^2} - 1)\right] &\le  \int_{\Omega(r^2/\eps+\alpha r)}^\infty \frac{1}{\sqrt{2 \pi r^2}} \frac{z^2}{r^4} \, e^{\frac{\eps z}{r^2}} e^{-\frac{z^2}{2r^2}} \,\d z\\
      &= e^{\frac{\eps^2}{2r^2}}\int_{\Omega(r^2/\eps+\alpha r)}^\infty \frac{1}{\sqrt{2 \pi r^2}} \frac{z^2}{r^4} \, e^{-\frac{(z-\eps)^2}{2r^2}} \,\d z\\
      &\lesssim \int_{\Omega(r^2/\eps+\alpha r)}^\infty \frac{1}{\sqrt{2 \pi r^2}} \frac{z^2}{r^4} \, e^{-\frac{z^2}{2r^2}} \,\d z \quad \text{since $\eps \le r/2$}\\
      &= \frac{1}{r^2} \int_{\Omega(r/\eps+\alpha)}^\infty \frac{1}{\sqrt{2 \pi}} z^2 \, e^{-\frac{z^2}{2}} \,\d z\\
      &\lesssim \frac{1}{r^2}O\left(e^{-\Omega(r^2/\eps^2 + \alpha^2)}\right) \lesssim \frac{1}{r^2} \frac{\eps}{r}e^{-\Omega(\alpha^2)}
  \end{align*}
  
  Thus, we have established that
  \begin{align*}
    \E_{x\sim f_r}\left[ s_r^2(x + \eps)\right] - \I_r \lesssim \frac{\eps \alpha}{r}\left(\I_r + \frac{ \alpha^2}{r^2} e^{-\alpha^2/2} + \frac{1}{r^2}e^{-\Omega(\alpha^2)}\right)
  \end{align*}
  Set $\alpha = O(\sqrt{ \log \frac{1}{r^2\I_r}})$ and recall from Lemma~\ref{lem:IrBounded} that $\I_r \le 1/r^2$.
  Therefore, the first term dominates the right hand side, and the lemma follows.
  
\end{proof}

With the above lemmas, we are now ready to prove Lemma~\ref{lem:localMLE}, which we also restate here for the reader's convenience.

\LemLocalMLE*

\begin{proof}
  Without loss of generality, we only show the $\lambda - \eps$ case, and the $\lambda+\eps$ case follows by doubling the failure probability.
  
  Combining Corollary~\ref{cor:scoreestimate}, Lemmas~\ref{lem:grad_expectation_theta} and~\ref{lem:var_close_to_fisher} as well as the assumption that $r^2\sqrt{\I_r} \ge \gamma \eps_{\max}$, we have that, for all $0 < \eps < \min(|r|,\eps_{\max})$, with probability at most $\delta$, we have
  \begin{align*}
      \hat{s}(\lambda-\eps) - \left(- \I_r\eps + \beta\sqrt{\I_r}\frac{\eps^2}{r^2}\right) &\le \sqrt{\frac{2\log\frac{1}{\delta}}{n} \I_r}\sqrt{1+\eta\frac{\eps}{r}\sqrt{\log\frac{1}{r^2 \I_r}}} + \frac{15\log\frac{1}{\delta}}{nr}\\
      &\le \sqrt{\frac{2\log\frac{1}{\delta}}{n} \I_r}\sqrt{1+\eta\frac{\eps_{\max}}{r}\sqrt{\log\frac{1}{r^2 \I_r}}} + \frac{15\log\frac{1}{\delta}}{nr}\\
      &\le \sqrt{\frac{2\log\frac{1}{\delta}}{n} \I_r}\sqrt{1+\eta\frac{\eps_{\max}}{r}\frac{1}{\sqrt{r^2 \I_r}}} + \frac{15\log\frac{1}{\delta}}{nr}\\
      &\le \sqrt{1+\frac{\eta}{\gamma}}\sqrt{\frac{2\log\frac{1}{\delta}}{n} \I_r} + \frac{15\log\frac{1}{\delta}}{nr}\\
      &\le \sqrt{1+\frac{\eta}{\gamma}}\sqrt{\frac{2\log\frac{1}{\delta}}{n} \I_r} + \frac{15}{2\sqrt{\gamma}}\left(\frac{2\log\frac{1}{\delta}}{n}\right)^{\frac{1}{4}}\sqrt{\frac{2\log\frac{1}{\delta}}{n} \I_r}\\
      &= \left(\sqrt{1+\frac{\eta}{\gamma}} + \frac{15}{2\sqrt{\gamma}}\left(\frac{2\log\frac{1}{\delta}}{n}\right)^{\frac{1}{4}} \right)\sqrt{\frac{2\log\frac{1}{\delta}}{n} \I_r}
  \end{align*}
  where the last inequality is due to the assumption that $r^2\sqrt{\I_r} \ge \gamma \eps_{\max} \ge \gamma \sqrt{\frac{2\log\frac{1}{\delta}}{n} \I_r}$.
  
  For the rest of the proof, we denote the multiplicative term $\sqrt{1+\frac{\eta}{\gamma}} + \frac{15}{2\sqrt{\gamma}}\left(\frac{2\log\frac{1}{\delta}}{n}\right)^{\frac{1}{4}}$ simply by $1+\rho$, as defined in the theorem statement.
  
  We note also that, since $r^2\sqrt{\I_r} \ge \gamma \eps$, we have $-\I_r\eps + \beta\sqrt{\I_r}\frac{\eps^2}{r^2} \le \left(-1+\frac{\beta}{\gamma}\right)\I_r\eps$.
  Therefore, for any $0 < \eps < \min(|r|,\eps_{\max})$, we have that with probability at least $1-\delta$,
  \[ \hat{s}(\lambda-\eps) \le \left(-1+\frac{\beta}{\gamma}\right)\I_r\eps + (1+\rho_r)\sqrt{\frac{2\log\frac{1}{\delta}}{n} \I_r}\]
  
  

  By Lemma~\ref{lem:u-lowerderivative}, we also have that for \emph{any} $x$
  \[ \hat{s}'(x) \le \frac{1}{r^2} \]
  
  Let $\xi \ll 1$ be a parameter that we choose at the end of the proof.
  We will show that with probability at least $1-\delta \cdot \frac{1}{\xi r^2\I_r(1-\frac{\beta}{\gamma})(1-\delta)}$, we have for all $\eps \in \left((1+\xi)\frac{1+\rho_r}{1-\frac{\beta}{\gamma}} \sqrt{\frac{2\log\frac{1}{\delta}}{n \I_r}}, \eps_{\max}\right]$, $\hat{s}(\lambda-\eps) < 0$.
  
  Consider a net $N$ of spacing $\xi r^2(1+\rho_r) \sqrt{\frac{2\log\frac{1}{\delta}}{n}\I_r}$ over the interval $\left((1+\xi)\frac{1+\rho_r}{1-\frac{\beta}{\gamma}} \sqrt{\frac{2\log\frac{1}{\delta}}{n \I_r}}, \eps_{\max}\right]$ in the theorem statement.
  We can check that, if for all points $\eps \in N$, we have
  \begin{equation}
  \label{eq:MLE_net_point}
      \hat{s}(\lambda-\eps) \le -\xi(1+\rho_r) \sqrt{\frac{2\log\frac{1}{\delta}}{n}\I_r}
  \end{equation}
  then, because $\hat{s}' \le 1/r^2$, we have $\hat{s}(x) < 0$ for all $x \in \lambda+\left((1+\xi)\frac{1+\rho_r}{1-\frac{\beta}{\gamma}} \sqrt{\frac{2\log\frac{1}{\delta}}{n \I_r}}, \eps_{\max}\right]$.
  This is done by considering two consecutive net points $0 < \eps_1 < \eps_2$, and observing that $\hat{s}(\lambda-\eps) \le \hat{s}(\lambda-\eps_1) - \frac{\eps_1-\eps}{r^2}$ for $\eps \in (\eps_1,\eps_2]$, which is in turn strictly negative.
  (For the essentially symmetric case of $\lambda+\eps$, we would instead use the inequality that $\hat{s}(\lambda+\eps) \ge \hat{s}(\lambda+\eps_1) + \frac{\eps_1-\eps}{r^2} > 0$.)
  
  Thus, it suffices to bound the probability that the above inequality holds for all points in $N$.
  For a natural number $i \ge 1$, consider the subset $N_i$ of $N$ that intersects with $([i,i+1]+\xi)\frac{1+\rho_r}{1-\frac{\beta}{\gamma}} \sqrt{\frac{2\log\frac{1}{\delta}}{n \I_r}}$, where here we interpret addition and multiplication as scalar operations on every point in the interval.
  For each $\eps \in N_i$, Equation~\ref{eq:MLE_net_point} holds except for probability $\delta^{i^2}$.
  Furthermore, each $N_i$ consists of $\frac{1+\rho_r}{1-\frac{\beta}{\gamma}} \sqrt{\frac{2\log\frac{1}{\delta}}{n \I_r}}$ divided by $\xi r^2(1+\rho_r) \sqrt{\frac{2\log\frac{1}{\delta}}{n}\I_r}$ many points, which equals to $\frac{1}{\xi r^2\I_r(1-\frac{\beta}{\gamma})}$ many points.
  Therefore, the total failure probability is at most
  $\frac{1}{\xi r^2\I_r(1-\frac{\beta}{\gamma})} \cdot \sum_{i \ge 1} \delta^{i^2} \le \delta \cdot \frac{1}{\xi r^2\I_r(1-\frac{\beta}{\gamma})(1-\delta)}$.
  An extra factor of 2 in the failure probability in the theorem statement accounts for the symmetric case of $\hat{s}(\lambda+\eps) > 0$.
\end{proof}

%% file: app_lb.tex
\section{Proof of Theorem~\ref{thm:lowerbound} in Section~\ref{sec:lb}}
\label{app:lb}

The goal of this appendix is to prove Theorem~\ref{thm:lowerbound}, which we restate here for the reader's convenience.

\ThmLB*

We use the standard proof technique of reducing distinguishing two ``close" distributions to estimation.
In particular, we show that it is statistically impossible to distinguish between $f_r$ and $f_r^{2\eps}$ with probability $1-\delta$ using $n$ samples.
In order to show such an indistinguishability result, we need the following standard fact (essentially the Neyman-Pearson lemma):

\begin{fact}
\label{fact:twopoint}
Consider a game, where an adversary picks arbitrarily either distribution $\vec{p}$ or distribution $\vec{q}$, and we want an algorithm which, on input $n$ independent samples from the chosen distribution, decide whether the samples came from $\vec{p}$ or $\vec{q}$, succeeding with probability at least $1-\delta$.
Then, there is no algorithm $\mathcal{A}$ such that:
\[\mathbb{P}(\mathcal{A}\text{ returns }\vec{p}\;|\;\text{adversary picked }\vec{p})-\mathbb{P}(\mathcal{A}\text{ returns }\vec{p}\;|\;\text{adversary picked }\vec{q})>\DTV(\vec{p}^{\otimes n},\vec{q}^{\otimes n})\]
where $\vec{p}^{\otimes n}$ denotes the $n$-fold product distribution of $\vec{p}$.
In particular, this implies that there is no algorithm $\mathcal{A}$ such that both of the following hold:
\begin{itemize}
    \item $\mathbb{P}(\mathcal{A}\text{ returns }\vec{p}\;|\;\text{adversary picked }\vec{p})>\frac{1}{2}+\frac{1}{2}\DTV(\vec{p}^{\otimes n},\vec{q}^{\otimes n})$
    \item $\mathbb{P}(\mathcal{A}\text{ returns }\vec{q}\;|\;\text{adversary picked }\vec{q})>\frac{1}{2}+\frac{1}{2}\DTV(\vec{p}^{\otimes n},\vec{q}^{\otimes n})$
\end{itemize}
So if $\DTV(\vec{p}^{\otimes n},\vec{q}^{\otimes n})<1-2\delta$, there is no algorithm that will succeed in distinguishing between two distributions with probability $\geq 1-\delta$ using only $n$ samples.
\end{fact}

Thus, we need to upper bound the $n$-sample total variation distance between $f_r$ and $f_r^{2\eps}$.
Standard inequalities for doing so involve calculating and plugging-in the single-sample KL-divergence $\kl{f_r}{f_r^{2\eps}}$ or squared Hellinger distance $\DH^2(f_r, f_r^{2\eps})$, however, they yield only constant factor tightness in the exponent of $1-\DTV(\vec{p}^{\otimes n},\vec{q}^{\otimes n})$, and hence only constant factor tightness in sample complexity or estimation error lower bounds.
As such, in this paper, we prove a new lemma (Lemma~\ref{lem:newlb}) that involves both the KL-divergence and squared Hellinger distance, as well as assumptions on the concentration of the log-likelihood ratio between $f_r$ and $f_r^{2\eps}$ (which will be satisfied by $r$-smoothed distributions), which allows us to bound the $n$-sample total variation distance tightly.
After that, we calculate the KL divergence and squared Hellinger distance of $f_r$ and $f_r^{2\eps}$ as well as show the concentration of their log likelihood ratio (Appendix~\ref{app:lbcalc}), which when applied to the lemma yields the lower bound result (Appendix~\ref{app:lbwrap}).

\begin{lemma}
\label{lem:newlb}
Consider two arbitrary distributions $p$, $q$.
Let the log-likelihood ratio be defined as $\gamma = \log\frac{q}{p}$.
  Suppose there is a parameter $\kappa > 0$ such that we have the following conditions:
  \begin{enumerate}
  \item $\DH^2(p, q) \leq (1+\kappa) \frac{1}{4}\kl{p}{q}$
  \item $\frac{\kl{p}{q}}{\kl{q}{p}} \in [(1+\kappa)^{-1},1+\kappa]$
  \item $\E_p[\gamma^2] \le (1+\kappa)2\kl{p}{q}$
  \item $\E_q[\gamma^2] \le (1+\kappa)2\kl{p}{q}$
  \item $\E_p[|\gamma|^k] \leq (1+\kappa) \frac{k!}{2} 2\kl{p}{q} \kappa^{k-2}$
  \item $\E_q[|\gamma|^k] \leq (1+\kappa) \frac{k!}{2} 2\kl{p}{q} \kappa^{k-2}$
  \end{enumerate}
  Then, for $\kappa \le 0.01$, $n\kl{p}{q}\gg 1$ and $\kl{p}{q} \ll 1$,
  \[
    1 - \DTV(p^{\otimes n}, q^{\otimes n}) \geq 2e^{-(1 + O(\kappa)+O(1/\sqrt{n \kl{p}{q}}) + O(\kl{p}{q}) )n \kl{p}{q}/4}
  \]
\end{lemma}
\begin{proof}

  Define
  \[
    BC_{S}(p, q) = \int_S \sqrt{pq} \leq \sqrt{p(S) q(S)}
  \]
  to be the restriction of the Bhattacharyya coefficient
  $BC(p, q) = 1 - \DH^2(p, q)$ to a subset $S$ of the domain.  For any
  $S$, we have
  \[
    1-\DTV(p, q)= \int \min(p, q) \geq \int_S \min(p, q) \geq \frac{(\int_S \sqrt{\min(p,q)\max(p,q)})^2}{\int_S \max(p, q)} = \frac{BC_S(p,q)^2}{p(S) + q(S)}.
  \]
  We apply this to $p^{\otimes n}$ and $q^{\otimes n}$, getting for any $S$:
  \begin{align}\label{eq:TVtoBC}
    1-\DTV(p^{\otimes n}, q^{\otimes n}) \geq \frac{BC_S(p^{\otimes n},q^{\otimes n})^2}{p^{\otimes n}(S) + q^{\otimes n}(S)}.
  \end{align}
  
  Thus, the goal now is to find an event $S$ such that $BC_S(p^{\otimes n},q^{\otimes n})$ is big relative to $p^{\otimes n}(S)+q^{\otimes n}(S)$.
  
  For the rest of the proof, we use the notation $\bar{\gamma}$ to denote the $n$-sample empirical log-likelihood ratio, namely $\frac{1}{n}\sum_i \gamma_i = \frac{1}{n}\sum_i \log\frac{q(x_i)}{p(x_i)}$.
  
  We now define $S_k$, for $k \in \Z$, to be the event
  $\{\overline{\gamma} \in [k-\frac{1}{2}, k+\frac{1}{2}]
  \cdot \alpha \kl{p}{q}\}$ for some parameter $\alpha = \Theta(\max(\kappa,1/\sqrt{n\kl{p}{q}},\kl{p}{q}))$, and set
  $S = S_0$.
  We have that
  \begin{align}
    \label{eq:BCtoh}
    BC_S(p^{\otimes n}, q^{\otimes n}) = BC(p^{\otimes n}, q^{\otimes n}) - \sum_{k \neq 0} BC_{S_k}(p^{\otimes n}, q^{\otimes n}) \geq BC(p, q)^n - \sum_{k \neq 0} \sqrt{p^{\otimes n}(S_k)q^{\otimes n}(S_k)}
  \end{align}

  Now, define
  \[
    \delta := e^{-n\min(\kl{p}{q}, \kl{q}{p})/4}
  \]
  We note that $\delta \le (BC(p,q)^n)^{(1-O(\kappa)-O(\DKL(p \, || \, q)))}$, as follows:
  \begin{align*}
      BC(p,q) &= 1-\DH^2(p,q)\\
      &\ge 1-(1+O(\kappa))\frac{1}{4}\min(\kl{p}{q}, \kl{q}{p})\\
      &\ge \exp\left(-(1+O(\kappa)+O(\kl{p}{q}))\frac{1}{4}\min(\kl{p}{q}, \kl{q}{p})\right)
  \end{align*}
  where the first inequality follows from conditions 1 and 2 in the lemma statement, and the second inequality follows from the fact that $1-x = \exp(-(1+\Theta(x))x)$.
  The above claim follows from raising both sides to the power of $n$.
  
  We shall now bound $p^{\otimes n}(S_k)$ and $q^{\otimes n}(S_k)$ in terms of $\delta$.
By standard sub-Gamma concentration bounds, conditions 3--6 imply that
  \begin{align}\label{eq:hgaussianp}
    \Pr_p\left[\overline{\gamma} > -\kl{p}{q} + t \sqrt{2(1+\kappa)\kl{p}{q}} + \frac{\kappa}{2}t^2\right] < e^{-n t^2/2}
  \end{align}
  and
  \begin{align}\label{eq:hgaussianq}
    \Pr_q\left[\overline{\gamma} < \kl{q}{p} - t \sqrt{2(1+\kappa)\kl{p}{q}}-\frac{\kappa}{2}t^2\right] < e^{-n t^2/2}
  \end{align}
  
  We now bound $p^{\otimes n}(S_k)$ by
  \begin{align*}
      p^{\otimes n}(S_k) &\leq \Pr_p\left[\overline{\gamma} \geq \left(k-\frac{1}{2}\right) \alpha \kl{p}{q}\right]
  \end{align*}
  Solving the equation
  \[ \frac{\kappa}{2}t_k^2 + \sqrt{2(1+\kappa)\kl{p}{q}}\,t_k -\kl{p}{q} = \left(k-\frac{1}{2}\right)\alpha\kl{p}{q} \]
  yields
  \[ t_k = \frac{\sqrt{1+\kappa}}{\kappa}\sqrt{2\kl{p}{q}}\left(\sqrt{1+\frac{\kappa}{1+\kappa}\left(1+\left(k-\frac{1}{2}\right)\alpha\right)}-1\right) \]
  By Equation~\ref{eq:hgaussianp}, $p^{\otimes n}(S_k) \le e^{-nt_k^2/2}$ whenever $(k-\frac{1}{2})\alpha > -1$.
  
  Also observe that, when $\kappa \le 0.01$, the function $\frac{1}{\kappa}(\sqrt{1+\frac{\kappa}{1+\kappa}(1+x)}-1)$ within the range $x \in [-1, 1.01]$ can be lower bounded by simply $\frac{1}{2}(1-2\kappa)(1+x)$.
  For the range $x \ge 1$, we can lower bound the function by $(1-2\kappa)\sqrt{x}$.
  This implies that $p^{\otimes n}(S_k)$ can be upper bounded by $e^{-\frac{n}{4}(1-O(\kappa))\kl{p}{q}(1+(k-\frac{1}{2})\alpha)^2} \le \delta^{(1-O(\kappa))(1+(k-\frac{1}{2})\alpha)^2}$ when $(k-\frac{1}{2})\alpha \in [-1, 1.01]$, and similarly, upper bounded by $e^{-n(1-O(\kappa))\kl{p}{q}(k-\frac{1}{2})\alpha} \le \delta^{4(1-O(\kappa))(k-\frac{1}{2})\alpha}$ when $(k-\frac{1}{2})\alpha \ge 1$.
  Finally, observe that for $(k-\frac{1}{2})\alpha \le -1$, we can trivially bound $p^{\otimes n}(S_k)$ by 1.
  
  We now bound $q^{\otimes n}(S_k)$ by
  \begin{align*}
      q^{\otimes n}(S_k) &\leq \Pr_q\left[\overline{\gamma} \leq \left(k+\frac{1}{2}\right) \alpha \kl{p}{q}\right]
  \end{align*}
  Solving the equation
  \[ -\frac{\kappa}{2}(t'_k)^2 - \sqrt{2(1+\kappa)\kl{p}{q}}\,t'_k +\kl{q}{p} = \left(k+\frac{1}{2}\right)\alpha\kl{p}{q} \]
  yields (by condition 2)
  \[ t'_k \ge \frac{\sqrt{1+\kappa}}{\kappa}\sqrt{2\kl{p}{q}}\left(\sqrt{1+\frac{\kappa}{1+\kappa}\left(1-\kappa-\left(k+\frac{1}{2}\right)\alpha\right)}-1\right) \]
  By Equation~\ref{eq:hgaussianq}, $q^{\otimes n}(S_k) \le e^{-n(t'_k)^2/2}$ whenever $(k+\frac{1}{2})\alpha < 1-\kappa$.
  
  We now bound $q^{\otimes n}(S_k)$ similar to how we bounded $p^{\otimes n}(S_k)$.
  When $\kappa \le 0.01$, the function $\frac{1}{\kappa}(\sqrt{1+\frac{\kappa}{1+\kappa}(1-\kappa+x)}-1)$ within the range $x \in [-1.01, 1-\kappa]$ can be lower bounded by simply $\frac{1}{2}(1-2\kappa)(1-\kappa-x)$.
  For the range $x \le -1$, we can lower bound the function by $(1-2\kappa)\sqrt{x}$.
  This implies that $q^{\otimes n}(S_k)$ can be upper bounded by $e^{-\frac{n}{4}(1-O(\kappa))\kl{p}{q}(1-\kappa-(k+\frac{1}{2})\alpha)^2} \le \delta^{(1-O(\kappa))(1-\kappa-(k+\frac{1}{2})\alpha)^2}$ when $(k+\frac{1}{2})\alpha \in [-1.01, 1-\kappa]$, and similarly, upper bounded by $e^{n(1-O(\kappa))\kl{p}{q}(k+\frac{1}{2})\alpha} \le \delta^{-4(1-O(\kappa))(k+\frac{1}{2})\alpha}$ when $(k+\frac{1}{2})\alpha \le -1$.
  Finally, observe that for $(k+\frac{1}{2})\alpha \ge 1-\kappa$, we can trivially bound $q^{\otimes n}(S_k)$ by 1.

  We are now ready to upper bound $\sum_{k \neq 0}\sqrt{p^{\otimes n}(S_k)q^{\otimes n}(S_k)}$.
  We decompose this sum into three regions of non-zero $k$.
  
  The main region $K_1$ is where $(k-\frac{1}{2})\alpha \ge -1$ and $(k+\frac{1}{2})\alpha \le 1-\kappa$.
  In this region, $\sqrt{p^{\otimes n}(S_k)q^{\otimes n}(S_k)}$ can be upper bounded by $\delta^{\frac{1-O(\kappa)}{2}\left[(1+(k-\frac{1}{2})\alpha)^2+(1-\kappa-(k+\frac{1}{2})\alpha)^2\right]} \le \delta^{(1-O(\kappa))(1+(|k|-\frac{1}{2})^2\alpha^2))}$.
  Now observe that $\sum_{k\in K_1}\delta^{(1-O(\kappa))(1+(|k|-\frac{1}{2})^2\alpha^2))} \lesssim \delta^{1-O(\kappa)+\Omega(\alpha)} \le \delta^{1+\Omega(\alpha)}$ as long as $\delta^{O(\alpha^2)} \ll 1$ and $\alpha = \Omega(\kappa)$.
  These two conditions are satisfied by the choice of $\alpha = \Omega(\max(\kappa, 1/\sqrt{n\kl{p}{q}}))$.
  
  The second region $K_2$ is where $(k-\frac{1}{2})\alpha \le -1$, which also means that $(k+\frac{1}{2})\alpha \le -(1+\Omega(\alpha))$.
  In this case, we use the bound $p^{\otimes n}(S_k) \le 1$ and $q^{\otimes n}(S_k) \le \delta^{-4(1-O(\kappa))(k+\frac{1}{2})\alpha}$.
  Thus, in this region, $\sqrt{p^{\otimes n}(S_k)q^{\otimes n}(S_k)}$ is upper bounded $\delta^{-2(1-O(\kappa))(k+\frac{1}{2})\alpha}$.
  This means that $\sum_{k \in K_2} \sqrt{p^{\otimes n}(S_k)q^{\otimes n}(S_k)} \lesssim \delta^{2(1-O(\kappa))(1-O(\alpha))} \ll \delta^{1+\Omega(\alpha)}$, as long as $\delta^\alpha \ll 1$ and as long as $\kappa \ll 1$ and $\alpha \ll 1$.
  This is again satisfied by our choice of $\alpha$.
  
  The last region $K_3$ is where $(k+\frac{1}{2})\alpha \ge 1-\kappa$, which also means that $(k-\frac{1}{2})\alpha \ge (1-O(\kappa)-O(\alpha))$.
  In this case, we use the bound $p^{\otimes n}(S_k) \le \delta^{4(1-O(\kappa))(k-\frac{1}{2})\alpha}$ and $q^{\otimes n}(S_k) \le 1$.
  Thus, in this region, $\sqrt{p^{\otimes n}(S_k)q^{\otimes n}(S_k)}$ is upper bounded $\delta^{2(1-O(\kappa))(k-\frac{1}{2})\alpha}$.
  This means that $\sum_{k \in K_3} \sqrt{p^{\otimes n}(S_k)q^{\otimes n}(S_k)} \lesssim \delta^{2(1-O(\kappa))(1-O(\alpha))} \ll \delta^{1+\Omega(\alpha)}$, as long as $\delta^\alpha \ll 1$ and as long as $\kappa \ll 1$ and $\alpha \ll 1$.
  This is again satisfied by our choice of $\alpha$.
  
  Summarizing, we have shown that $\sum_{k \ne 0} \sqrt{p^{\otimes n}(S_k)q^{\otimes n}(S_k)} \lesssim \delta^{1+\Omega(\alpha)}$ for $\alpha = \Omega(\max(\kappa, 1/\sqrt{n\kl{p}{q}}))$.
  Furthermore, as $\alpha \ge \Omega(\kl{p}{q})$ by construction, the above bound is much less than $(BC(p,q))^n$, since $(BC(p,q))^n \ge \delta^{1+O(\kappa)+O(\kl{p}{q})}$ from earlier in this proof.
  This yields that $BC_S(p^{\otimes n},q^{\otimes n}) \ge BC(p,q)^n - \sum_{k \ne 0}\sqrt{p^{\otimes n}(S_k)q^{\otimes n}(S_k)} \ge \delta^{1+O(\alpha)}$ since $\alpha = \Omega(\kl{p}{q})$.
  
  The last quantities we have to bound are $p^{\otimes n}(S_0)$ and $q^{\otimes n}(S_0)$.
  These were already bounded in the respective paragraphs bounding $p^{\otimes n}(S_k)$ and $q^{\otimes n}(S_k)$ for general $k$.
  When $k = 0$, the bounds are at most $\delta^{1+\Omega(\alpha)}$ (again, when $\alpha = \Omega(\kappa)$).
  Finally, we get that
  \[ 1-\DTV(p^{\otimes n},q^{\otimes n}) \ge \frac{(BC_S(p^{\otimes n},q^{\otimes n}))^2}{p^{\otimes n}(S)+q^{\otimes n}(S)} \gtrsim \frac{\delta^{2+O(\alpha)}}{\delta^{1+\Omega(\alpha)}} = \delta^{1+O(\alpha)}
  \]
  Expanding the definition of $\delta$ as well as the choice of $\alpha = \Theta(\max(\kappa, 1/\sqrt{n\kl{p}{q}}, \kl{p}{q}))$ gives the lemma statement.
\end{proof}

\subsection{Showing the conditions for Lemma~\ref{lem:newlb}}
\label{app:lbcalc}

In this subsection, we calculate the KL-divergence, squared Hellinger distance, as well as moment bounds for the log-likelihood ratio for $f_r$ and $f_r^{2\eps}$ for a generic $r$-smoothed distribution $f_r$.

\begin{lemma}\label{lem:kl-bound}
  	Consider the parametric family $f^{\lambda}_r(x) = f_r(x-\lambda)$ for some $r$-smoothed distribution $f_r$ with Fisher information $\I_r$.
	Then for $\eps<\frac{r}{4}$,
	\[ \kl{f_r}{f_r^{2\eps}}= 2\eps^2 \I_r \left(1+\Theta\left(\frac{\eps}{r^2\sqrt{\I_r}}\right)\,\right) \]
\end{lemma}

\begin{proof}
Let $\ell(x)=\log f_r(x)$
\begin{align*}
	\kl{f_r}{f_r^{2\eps}} &= -\int_{-\infty}^{\infty} f_r(x) \log \frac{f_r^{2\eps}(x)}{f_r(x)} \, \d x\\
	&=-\int_{-\infty}^{\infty} f_0(x) \int_{x}^{x+2\eps}\ell'(y)\,\d y\\
	&=-\int_{0}^{2\eps}\E_{z\from f_r}\left[s_r(z+y)\right]\,\d y\\
	&=\int_{0}^{2\eps}\I_r y + \Theta\left(\sqrt{\I_r} \frac{y^2}{r^2}\right)\,\d y\\
	&=2\eps^2 \I_r + \Theta\left(\sqrt{\I_r}\frac{\eps^3}{r^2}\right) = 2\eps^2 \I_r \left(1+\Theta\left(\frac{\eps}{r^2\sqrt{\I_r}}\right)\,\right)
	\end{align*}
where the $\Theta$ result is from Lemma~\ref{lem:grad_expectation_theta}.

  
\end{proof}

\begin{lemma}
\label{lem:h2bound}
	Consider the parametric family $f^{\lambda}_r(x) = f_r(x-\lambda)$ for some $r$-smoothed distribution $f_r$ with Fisher information $\I_r$.
	Then for $\eps\leq r$,
	\[ \DH^2(f_r, f_r^{2\eps})\leq  \frac{1}{4}\kl{f_r}{f_r^{2\eps}}+O(\frac{\eps^3}{r^3}) \]
\end{lemma}
\begin{proof}
  	Observe that the squared Hellinger distance is
	\begin{align*}
	\DH^2(f_r, f_r^{2\eps}) &= \frac{1}{2}\int_{-\infty}^{\infty} (\sqrt{f_r(x)} - \sqrt{f_r^{2\eps}(x)})^2 \, \d x\\
	&= \frac{1}{2}\int_{-\infty}^{\infty} f_r(x) (1 - \sqrt{\frac{f_r^{2\eps}(x)}{f_r(x)}})^2 \, \d x
	\end{align*}
Since for $y>0$ we have $(1-\sqrt{y})^2\leq -\frac{1}{2}(\log y) + \frac{y-1}{2}+\frac{1}{24}(y-1)_+^3$, substituting $y=\frac{f_r^{2\eps}(x)}{f_r(x)}$, the previous expression is at most
\[\frac{1}{2}\int_{-\infty}^{\infty} f_r(x) \left(-\frac{1}{2}\log \frac{f_r^{2\eps}(x)}{f_r(x)} +\frac{1}{2}\frac{f_r^{2\eps}(x)}{f_r(x)}-\frac{1}{2}+\frac{1}{24}\left(\frac{f_r^{2\eps}(x)}{f_r(x)}-1\right)_+^3\right) \, dx\]
The first term inside the big parentheses equals $\frac{1}{4}\kl{f_r}{f_r^{2\eps}}$
; the next two terms cancel out (since they each integrate all the probability mass of $f$); the final, cubic, term we bound now.

We start by bounding the cubic term for a Gaussian $g$ of standard deviation $r$:\[\int_{-\infty}^{\infty} g(x) \left(\frac{g(x+2\eps)}{g(x)}-1\right)_+^3 \, \d x=O(\frac{\eps^3}{r^3})\]
where the bound is easily computed from the closed form evaluation of the integral, valid while $\eps$ is bounded by some fixed multiple of $r$.

Now the $r$-smoothed distribution $f$ is just a convex combination of Gaussians of width $r$, and the desired inequality follows from the observation that the expression $f_r(x)\left(\frac{f_r^{2\eps}(x)}{f_r(x)}-1\right)_+^3$ is convex in the sense that, in terms of $y,z>0$, the function $y\left(\frac{z}{y}-1\right)_+^3 $ is a convex 2-variable function. Thus the total contribution of the  cubic term is bounded by its value for the Gaussian, namely $O(\frac{\eps^3}{r^3})$.

\end{proof}

\begin{lemma}
\label{lem:gammakthmoment}
Let $k \ge 3$.
  For an $r$-smoothed distribution $f_r$, let $\gamma = \log \frac{f_r^{2\eps}}{f_r}$.
  For $\eps\leq r$, we have \[\E_p[|\gamma|^k]\leq\frac{k!}{2}(30\eps/r)^{k-2} 4\eps^2 \I \left( 1 + O\left( \frac{\eps}{r} \sqrt{\log \frac{1}{r^2\I}}\right)\right)\]
\end{lemma}
\begin{proof}

  Let $\ell(x)=\log f_r(x)$. We have
  \begin{align*}
      \int_{-\infty}^{\infty} f_r(x)\left|\log f_r(x)-\log f_r(x+2\eps)\right|^k&=\int_{-\infty}^{\infty}p(x)\left|\int_{x}^{x+2\eps}\ell'(y)\,\d y\right|^k\,\d x\\
      &\leq \eps^{k-1}\int_{-\infty}^{\infty}p(x)\int_{x}^{x+2\eps}|\ell'(y)|^k\,\d y\,\d x\\
      &= (2\eps)^{k-1}\int_{0}^\eps \E_x[\abs{s_r(x+y)}^k] \, \d y\\
      &\le (2\eps)^{k-1}\frac{k!}{2}(15/r)^{k-2} \int_0^{2\eps} \E_x[s_r(x+y)^2] \, \d y\\
      &\le (2\eps)^{k-1}\frac{k!}{2}(15/r)^{k-2} \int_0^{2\eps} \I + O\left( \frac{y}{r} \I \sqrt{\log \frac{1}{r^2\I}}\right)\, \d y\\
      &= (2\eps)^{k}\frac{k!}{2}(15/r)^{k-2} \I\left( 1 + O\left( \frac{\eps}{r} \sqrt{\log \frac{1}{r^2\I}}\right)\right)\\
      &= \frac{k!}{2}(30\eps/r)^{k-2} 4\eps^2 \I \left( 1 + O\left( \frac{\eps}{r} \sqrt{\log \frac{1}{r^2\I}}\right)\right)
  \end{align*}
where the first three inequalities are by convexity, by Lemma~\ref{lem:offcenter-score-moments}, and by Lemma~\ref{lem:var_close_to_fisher}.
\end{proof}

\begin{lemma}
\label{lem:gamma2ndmoment}
  For an $r$-smoothed distribution $f_r$, let $\gamma = \log \frac{f_r^{2\eps}}{f_r}$.
  For $\eps\leq r$, we have \[\E_{f_r}[|\gamma|^2]\leq4\eps^2 \I \left( 1 + O\left( \frac{\eps}{r} \sqrt{\log \frac{1}{r^2\I}}\right)\right)\]
\end{lemma}
\begin{proof}
  Let $\ell(x)=\log f_r(x)$. We have \begin{align*}
      \int_{-\infty}^{\infty} f_r(x)\left|\log f_r(x)-\log f_r(x+2\eps)\right|^2&=\int_{-\infty}^{\infty}f_r(x)\left|\int_{x}^{x+2\eps}\ell'(y)\,\d y\right|^2\,\d x\\
      &\leq \eps\int_{-\infty}^{\infty}f_r(x)\int_{x}^{x+\eps}|\ell'(y)|^2\,\d y\,\d x\\
      &= 2\eps \int_0^{2\eps} \E_x[s_r(x+y)^2] \, \d y\\
      &\le 2\eps \int_0^{2\eps} \I + O\left( \frac{y}{r} \I \sqrt{\log \frac{1}{r^2\I}}\right)\, \d y\\
      &= 4\eps^{2} \I\left( 1 + O\left( \frac{\eps}{r} \sqrt{\log \frac{1}{r^2\I}}\right)\right)
  \end{align*}
where the first two inequalities are by convexity, and by Lemma~\ref{lem:var_close_to_fisher}.
\end{proof}

\subsection{Proving Theorem~\ref{thm:lowerbound}}
\label{app:lbwrap}

We are ready to prove Theorem~\ref{thm:lowerbound}.

\begin{proof}[Proof of Theorem~\ref{thm:lowerbound}]
We will be applying Lemma~\ref{lem:newlb} on the distributions $f_r$ and $f_r^{2\eps}$ for an appropriately chosen $\eps$, with $\kappa = O(\frac{\eps}{r}\frac{1}{r^2\I_r})$ (note that by Lemma~\ref{lem:IrBounded}, $\I_r \le 1/r^2$ so $r^2\I_r \le 1$).

Lemma~\ref{lem:h2bound} combined with Lemma~\ref{lem:kl-bound} show condition 1 on Lemma~\ref{lem:newlb}.
Lemma~\ref{lem:kl-bound} shows condition 2.
Lemma~\ref{lem:gamma2ndmoment} shows condition 3, and an essentially identical calculation shows condition 4.
Lemma~\ref{lem:gammakthmoment} shows condition 5, and again an essentially identical calculation shows condition 6.

Thus, applying Lemma~\ref{lem:newlb} and Fact~\ref{fact:twopoint}, the failure probability of distinguishing $p = f_r$ and $q = f_r^{2\eps}$ is at least
\[ e^{-(1 + O(\frac{\eps}{r}\frac{1}{r^2\I_r})+O(1/\sqrt{n \eps^2\I_r}) + O(\eps^2\I_r) )n \eps^2\I_r/2} \]

Picking \[ \eps = \left(1-O\left(\sqrt{\frac{\log\frac{1}{\delta}}{n}}\frac{1}{r^3\I_r^{1.5}}\right)-O\left(\frac{1}{\log\frac{1}{\delta}}\right) - O\left(\frac{\log\frac{1}{\delta}}{n}\right)\right)\sqrt{\frac{2\log\frac{1}{\delta}}{n\I_r}}\]
yields a failure probability lower bound of $\delta$, thus showing the theorem statement.
\end{proof}



